\documentclass[11pt]{article}
\usepackage{amsmath, amsthm, amssymb, amsfonts}
\usepackage{geometry}
\usepackage{hyperref}
\usepackage{booktabs}
\usepackage{graphicx}
\usepackage{placeins}

\newtheorem{theorem}{Theorem}
\newtheorem{proposition}[theorem]{Proposition}

\newtheorem{lemma}[theorem]{Lemma}
\newtheorem{remark}{Remark}
\newtheorem{definition}{Definition}
\newcommand{\eps}{\varepsilon}
\newcommand{\bOmega}{{\partial\Omega}}
\newcommand{\R}{\mathbb{R}}
\newcommand{\T}{\mathbb{T}}

\newcommand{\aeps}{a(x/\eps)}
\newcommand{\astar}{a^*}
\newcommand{\teps}{\theta^\eps}
\newcommand{\tstar}{\theta^*}
\newcommand{\ttilde}{\tilde\theta}
\newcommand{\combo}{\ttilde + \tstar}
\newcommand{\chin}{\chi_n}
\newcommand{\dnu}{\partial_n u^0}

\newcommand{\dv}{\,\mathrm{d}}
\newcommand{\Lp}{L_\perp(n)}
\newcommand{\Nn}{N_n}
\newcommand{\G}{G}
\newcommand{\codeurl}{https://github.com/sharimoskow/bltail}

\title{A cell problem for the boundary density\\[4pt]
       and approximation of boundary correctors in homogenization theory}
\author{S.~Moskow}
\date{}

\begin{document}
\maketitle

\begin{abstract}
We study the boundary corrector $\teps$ in the two-scale expansion for
periodic elliptic problems with Dirichlet conditions.  By results of
G\'{e}rard-Varet--Masmoudi and Prange, on smooth uniformly convex domains
$\teps$ converges to the solution $\tstar$ of a homogenized problem with
boundary data given by the boundary layer tail $d(n)$, which is determined by
a half-space problem for each normal $n$.  In the lifted form of
G\'{e}rard-Varet and Masmoudi, $d(n)$ is given explicitly by the cell
correctors and the Dirichlet-to-Neumann operator $N_n$ of this half-space problem.  
Our main result is that $N_n$ satisfies a nonlinear equation on the period
cell alone, so that $d(n)$ and the boundary density that represents it can
be computed on the cell. This equation is a Riccati equation in which the derivative in depth is replaced by a commutator with the
derivative along $n$.  This cell problem can therefore be solved by a Riccati doubling algorithm, and we can compute $d(n)$ at irrational normals without rational
approximation.  Dropping the commutator term in the cell equation gives an operator
square root  approximation for $N_n$ 
that captures the leading term of $d(n)$ for small contrast, for the truncated problem.  Finally, we
approximate $\teps$ at finite $\eps$ by the sum of two solutions of constant
coefficient problems, one with oscillatory boundary data and the other
$\tstar$, and test it on the disk.
\end{abstract}

\section{Introduction}

In the two-scale expansion~\cite{BLP78} for elliptic problems with periodically
oscillating coefficients $a(x/\eps)$, the boundary corrector $\teps$ compensates
for the oscillatory mismatch at the boundary of the domain.  While the interior part of
the expansion is well understood, $\teps$ is both analytically subtle and
as computationally expensive as the original problem.  Resolving it with standard finite elements requires a mesh of size $h\ll\eps$
throughout the $d$-dimensional domain $\Omega$, a cost that grows as $O(\eps^{-d})$.
On smooth uniformly convex domains (or domains of finite type) $\teps$ converges as $\eps\to0$ to a limit $\tstar$, which solves the
homogenized equation with boundary data $d(n(x))\,\dnu(x)$, where $u^0$ is the
homogenized solution, $n(x)$ the outward normal, and $d(n)$ the
\emph{boundary layer tail}, a number attached to each unit normal.

For sides of rational normal the limit was established by
Moskow--Vogelius~\cite{MV96} and Allaire--Amar~\cite{AA99}; there it depends on
the sequence $\eps\to0$ through the position of each side relative to the
periodicity cell, so it is not unique, a dependence first observed by Santosa
and Vogelius~\cite{SV93} in their study of first-order eigenvalue corrections.  For smooth domains, G\'{e}rard-Varet
and Masmoudi showed that the tail $d(n)$ is well defined for normals
satisfying a Diophantine condition~\cite{GVM11} and that the whole family
$\teps$ converges on smooth uniformly convex domains~\cite{GVM12}, and
Prange~\cite{Prange13} showed that $d(n)$ exists for every irrational $n$; see
Section~\ref{sec:setup}.  On smooth uniformly convex domains, and more generally on domains of
finite type~\cite{Zhuge18}, the limit is therefore unique.  To our knowledge, a characterization of $d(n)$
on the period cell alone (rather than on the cell times a half-line), together with a practical way to compute it
directly at irrational normals, has remained open.  The limit comes from an interaction between
the oscillating coefficient and the oscillating boundary data in a boundary
layer of width $\eps$, and is not what one obtains by taking the two limits
separately; $d(n)$ has so far been available only through a half-space
problem for each normal.  Once $d(n)$ is known, $\tstar$ is obtained from one
constant coefficient problem.

For rational normals, $d(n)$ was characterized in~\cite{MV96} via a half-strip
analysis, in the form $$d=\int\chin\rho$$ with an explicit density $\rho$ and
$\chin$ the cell corrector in the normal direction, and in~\cite{AA99} through
the boundary layer tail of the half-strip corrector.  For irrational normals
Prange~\cite{Prange13} obtained, by an expansion of the Poisson kernel of the
half-space and an ergodic theorem, a representation of the same form.
(Boundary-layer correctors of the same half-strip type also govern the
effective transmission conditions across thin periodic interfaces, see
Delourme, Haddar and Joly~\cite{DHJ12}.)  In the lifted formulation of the half-space problem by G\'{e}rard-Varet and Masmoudi~\cite{GVM12},  the
unknown is a function of the depth $t$ and of a point $\theta$ of the
period cell $\T^d$.  The density is then given by 
\begin{equation}\label{eq:rho-intro}
  \rho(\cdot,n)=\frac{1}{A_{nn}}\Bigl(a\,(1-\partial_n\chin)+\Nn\chin\Bigr),
  \qquad A_{nn}:=n\cdot\astar n ,
\end{equation}
where $\astar$ is the homogenized matrix and $\Nn$ is the Dirichlet-to-Neumann
(DtN) operator of the lifted problem, which maps Dirichlet boundary data on $\T^d$ to
the conormal derivative of the corresponding boundary layer.  Everything in~\eqref{eq:rho-intro}
is an object on the cell except $\Nn$, whose definition involves the solution
of the half-space problem.

Our main result is that $\Nn$, and therefore the density, satisfies a
nonlinear equation on the period cell $\T^d$ in which the depth variable does not appear.  The equation involves two operators on $\T^d$, the derivative along $n$ and the transverse part of the elliptic operator
$-\nabla\cdot(a\nabla\,\cdot)$:
\begin{equation}\label{eq:GL-intro}
  \G:=n\cdot\nabla_y ,\qquad
  \Lp:=-\nabla^\perp\!\cdot\bigl(a\nabla^\perp\,\cdot\bigr),\qquad
  \nabla^\perp:=\nabla_y-n\,\partial_n .
\end{equation}
Theorem~\ref{thm:riccati} states that $\Nn$ is a symmetric, non-negative solution on $\T^d$ of
\begin{equation}\label{eq:CR-intro}
  \Nn\,a^{-1}\Nn \;+\; [\,\G,\Nn\,] \;=\; \Lp ,
  \qquad [\G,\Nn]:=\G\Nn-\Nn\G ,
\end{equation}
and Theorem~\ref{thm:density} that the density is the solution of the linear
problem
\begin{equation}\label{eq:kernel-intro}
  \Nn\Bigl(\frac{\rho}{a}\Bigr)+\partial_n\rho=0 ,\qquad \int_{\T^d}\rho=1 ,
\end{equation}
which says that $\rho$ is the density conserved by the depth evolution of the
boundary layer, and characterizes it without reference to the cell corrector.

Equation~\eqref{eq:CR-intro} is a Riccati equation in which the usual
derivative in the depth variable has been replaced by a commutator with
$\G=n\cdot\nabla_y$.
We also consider the approximation obtained by dropping the commutator,
\begin{equation}\label{eq:sqrt-intro}
  \Nn\;\approx\;a\sqrt{a^{-1}\Lp}\,,
\end{equation}
which is the operator square root that solves $Na^{-1}N=\Lp$.  This is equal to the DtN operator
for the problem where the coefficient is frozen along each normal line at its boundary
value.  For the Fourier-truncated problem at irrational normals, we show
that the resulting approximation reproduces the leading,
second-order term of $d(n)$ in the contrast of $a$, with a third-order error,
although the operators themselves differ at first order, and we test it
numerically.  

The Riccati structure also gives a numerical method.  Doubling is a
standard algorithm for algebraic Riccati
equations~\cite{Anderson78,CFLW04,BIM12}; it computes the solution as the
limit of finite depth problems whose depth doubles at each step, so that a few
solves reach a depth far beyond the decay length of the boundary layer.
Applied to the cell problem it computes $d(n)$ for irrational $n$ directly, with no phase averaging over rational approximants.  We validate the method against direct strip solves for two example coefficients, a smooth trigonometric one and a high-contrast
inclusion.  As a result, on such domains $\tstar$ can be computed by tabulating $d(n)$ over the circle of normals and solving one
constant coefficient problem.  The solver is available as an open-source
Python package, \texttt{bltail}~\cite{bltail}: given a periodic coefficient
$a$ in two dimensions and a list of normals, it returns $d(n)$.

With $d(n)$ in hand, we approximate the corrector at finite $\eps$ by
\begin{equation}\label{eq:combo-intro}
  \teps\;\approx\;\combo ,
\end{equation}
where both components solve problems with
the \emph{constant} homogenized operator $-\nabla\cdot(\astar\nabla)$ rather
than the full oscillatory operator $L_\eps$.  We define the tilde corrector $\ttilde$
to have the same oscillatory boundary data as $\teps$, but with the homogenized
interior coefficient, and the homogenized corrector $\tstar$ has the 
boundary data $d(n(x))\,\dnu(x)$ described above.  Neither component requires a fine interior
mesh.
The idea of keeping the microstructure only in the boundary data of a
problem with homogenized coefficients was used in~\cite{CMP20} (see
also~\cite{CGMP19}) for the Helmholtz equation with a periodic refractive
index, where the oscillation sits in a lower order term rather than in the principal part. In that case the boundary data of the limiting
boundary corrector is just the weak limit of its boundary data, and requires no
half-space problem. It vanishes on smooth domains and equals the average of
the restricted data on flat sides.  The same device was used for random media by Garnier, Giovangigli, Goepfert and
Millien~\cite{GGGM23}, who introduce a boundary layer corrector of this type
to compensate for the loss of stationarity at the boundary of a random
scatterer.  For oscillations in the principal part, as here, $\tstar$ is nontrivial
even on smooth domains, and its data $d(n)$ come from the boundary layer.
Note that the combination has the correct $L^2(\Omega)$ limit,
$\combo\to\tstar$ as $\eps\to 0$, on smooth uniformly convex domains (or
domains of finite type), since there $\tstar$ is the unique limit of $\teps$.  In numerical experiments on the unit disk the combination is considerably
more accurate than $\tstar$ alone (Section~\ref{sec:disk}).

The paper is organized as follows.
In Section~\ref{sec:setup} we set up the two-scale expansion and recall what is
known about the limit of $\teps$, including the half-strip description of
$d(n)$ for rational normals, and give a variational characterization of the
DtN operator (Lemma~\ref{lem:energy}).  In Section~\ref{sec:cell} we derive the cell problem
for $\Nn$ and the density, and the square root  approximation.
Section~\ref{sec:computing} contains the Riccati doubling algorithm and validation against
direct strip solves for two coefficients.  Sections~\ref{sec:setup}
to~\ref{sec:computing} contain the main results.  Section~\ref{sec:approx} is
an application: we use $d(n)$ to approximate $\teps$ at finite $\eps$ by
$\combo$, and Section~\ref{sec:disk} tests the approximation on the unit disk
for both coefficients.
Section~\ref{sec:conclusion} contains conclusions and future work. 

\section{Setup and known results}
\label{sec:setup}

\subsection{The two-scale expansion and the boundary corrector}

Let $\Omega\subset\R^d$ be a smooth bounded domain.  Consider the Dirichlet problem
\begin{equation}\label{eq:ueps}
  L_\eps u^\eps
  := -\nabla\cdot\bigl(\aeps\,\nabla u^\eps\bigr) = f
  \quad\text{in }\Omega,
  \qquad
  u^\eps = 0 \quad\text{on }\bOmega,
\end{equation}
where the coefficient $a(y)$ is a smooth, $\mathbb{Z}^d$-periodic
\emph{scalar} function with $a\ge\lambda>0$.
The standard two-scale ansatz~\cite{BLP78} gives the formal expansion
\begin{equation}\label{eq:expansion}
  u^\eps(x) \;\approx\;
  u^0(x)
  - \eps\,\chi^j(x/\eps)\,\partial_j u^0(x)
  + \eps\,\teps(x)
  + O(\eps^2),
\end{equation}
where repeated indices are summed, $u^0$ solves the homogenized problem
\begin{equation}\label{eq:u0}
  -\nabla\cdot(\astar\nabla u^0) = f \;\text{ in }\Omega,
  \quad u^0 = 0\;\text{ on }\bOmega,
\end{equation}
$\astar$ is the homogenized coefficient matrix with entries
$a^*_{ij} = \int_{\T^d} a(y)\bigl(\delta_{ij} - \partial_{y_i}\chi^j(y)\bigr)\dv y$,
and $\chi^j$ are the standard cell correctors solving
\begin{equation}\label{eq:cell}
  -\nabla_y\cdot\bigl(a(y)(e_j - \nabla_y\chi^j)\bigr) = 0
  \quad\text{on }\T^d, \qquad \int_{\T^d}\chi^j\dv y = 0.
\end{equation}

The \emph{boundary corrector} $\teps$ is introduced to restore the Dirichlet
condition, which is violated by the oscillatory bulk corrector term.  Since
$u^0 = 0$ on $\bOmega$, the tangential gradient of $u^0$ vanishes there, so
$\partial_j u^0\big|_{\bOmega} = n_j\,\dnu$ where $n$ is the outward unit normal
and $\dnu = \nabla u^0\cdot n$.  Writing $\chin = \chi^j n_j$ for the
cell corrector in the normal direction, $\teps$ satisfies
\begin{equation}\label{eq:teps}
  L_\eps\,\teps = 0 \;\text{ in }\Omega,
  \qquad
  \teps = \chin(x/\eps)\,\dnu(x)
  \;\text{ on }\bOmega.
\end{equation}
The boundary data $\chin(x/\eps)\dnu$ is $O(1)$ in $L^\infty$ but, because it
oscillates on the scale $\eps$, $O(\eps^{-1/2})$ in $H^{1/2}(\bOmega)$, so the
energy estimate for~\eqref{eq:teps} gives only
$\|\nabla\teps\|_{L^2(\Omega)} \le C\eps^{-1/2}$.  The gradient is
concentrated in a boundary layer of width $\sim\eps$, in which $\teps$
oscillates with amplitude $O(1)$.  As recalled next, $\teps$ converges in
$L^2(\Omega)$ to a nonzero limit $\tstar$.  If the oscillatory part
$\teps-\tstar$ were $O(1)$ only in a layer of width $\sim\eps$, its
$L^2(\Omega)$ norm would be $O(\eps^{1/2})$.  This is exact for a flat
boundary with rational normal, where the layer decays exponentially
(the upper bound on a rectangle is in~\cite[Lemma~3.4]{AA99}); for
irrational normals the layer can decay slowly into the
interior~\cite{Prange13}, and the proven rates on smooth domains are
weaker~\cite{GVM12,AKMP17}.
Computing $\teps$ directly with standard finite elements requires a mesh
with $h \ll \eps$ throughout $\Omega$
to resolve the oscillatory coefficients $a(x/\eps)$, which is prohibitively
expensive for small $\eps$.

\subsection{The limit of the boundary corrector}

Despite the boundary layer, $\teps$ converges in $L^2(\Omega)$ to a nonzero
limit $\tstar$, the solution of
\begin{equation}\label{eq:tstar}
  -\nabla\cdot(\astar\nabla\tstar) = 0 \;\text{ in }\Omega,
  \qquad \tstar = d(n(x))\,\dnu(x) \;\text{ on }\bOmega,
\end{equation}
where $d(n)$ is the \emph{boundary layer tail} in the direction $n$ and
$n(x)$ is the outward unit normal at $x\in\bOmega$.

The limit $\tstar$ is not what one obtains by taking the limits of the
operator and of the boundary data separately.  Replacing $a(x/\eps)$ by
$\astar$ and the data $\chin(x/\eps)\dnu$ by its weak limit in $L^2(\bOmega)$
would give zero, since $\int_{\T^d}\chin=0$ and the data average out along a
boundary whose normal is irrational almost everywhere.  The nonzero data
$d(n)\dnu$ of $\tstar$ is produced in the boundary layer of width $\eps$, where
the coefficient and the data oscillate on the same scale; the formula for
$d(n)$ recalled below makes this explicit through a density that depends on
$a$.

The existence of $d(n)$
and the convergence $\teps\to\tstar$ were established for rational normals by
Moskow--Vogelius~\cite{MV96}  via the half-strip analysis
recalled below, on polygonal domains all of whose sides have rational normal;
there it was also shown that $\tstar$ gives the first-order correction to the
homogenized solution, in the sense that $(u^\eps-u^0)/\eps\rightharpoonup\tstar$
weakly in $L^2(\Omega)$, and the first-order correction to the homogenized
eigenvalues, whose formula and dependence on the phase had been found by
Santosa and Vogelius~\cite{SV93}; see also Allaire--Amar~\cite{AA99}, who introduced the name
\emph{boundary layer tail} and gave a detailed analysis of the rational-normal
case.  For irrational normals the main results are those of G\'{e}rard-Varet and Masmoudi.  In~\cite{GVM11} they
showed that for normals satisfying a Diophantine (small divisors) condition,
which holds for almost every $n$, the half-space boundary layer converges to a tail $d(n)$ that does not depend on the position of the boundary,  so there is no subsequence
dependence.  In~\cite{GVM12} they used this, together with an approximation
of the boundary by polygons with Diophantine normals, to prove that on a
smooth uniformly convex domain $\Omega$ the whole family $\teps$ converges in
$L^2(\Omega)$ to the single limit $\tstar$ of~\eqref{eq:tstar}, with an
explicit rate; they remark that most of the analysis carries
over when the boundary has no flat piece with rational normal.  Prange~\cite{Prange13} then
showed, by a Poisson-kernel representation and an ergodic argument, that the
tail $d(n)$ exists for every irrational $n$, Diophantine or not.  Quantitative results, with
convergence rates and regularity of the homogenized boundary data, were
obtained by Aleksanyan, Shahgholian and Sj\"{o}lin~\cite{ASS15}, Armstrong,
Kuusi, Mourrat and Prange~\cite{AKMP17}, and Shen and Zhuge~\cite{SZ20}, and
Zhuge~\cite{Zhuge18} extended the convergence to smooth domains of finite
type, which need not be strictly convex.

\subsection{The strip problem for a rational normal}
\label{sec:strip}

Consider a flat side of $\Omega$ with outward normal $n=e_1$, say
$\{x_1=1\}$ with $\Omega$ on the side $x_1<1$.  In the fast variable $y=x/\eps$
the side is the line $y_1=1/\eps$.  Write
\begin{equation}\label{eq:phase}
  \frac1\eps=m+s,\qquad m\in\mathbb N,\quad s\in[0,1);
\end{equation}
by periodicity of $a$ the problem near the side is, up to the integer shift
$m$, a problem in the half-strip $\{y_1<s\}\times\T^1$, and the \emph{phase}
$s$ determines where the side cuts the periodicity cell.  Write $A_{11}=a^*_{11}$
for the $(1,1)$ entry of the homogenized matrix and let $\chi^1$ be the first
cell corrector.  Define the \emph{half-strip corrector} $\psi(\cdot;s)$ as the
$a$-harmonic solution with $\chi^1$ boundary data,
\begin{equation}\label{eq:psi}
\begin{gathered}
  -\nabla_y\cdot(a(y)\nabla_y\psi) = 0
  \;\text{ in }\{y_1<s\}\times\T^1,
  \qquad
  \psi(s,y_2) = \chi^1(s,y_2),\\
  \psi\to d(e_1;s) \;\text{ as }y_1\to-\infty,
\end{gathered}
\end{equation}
with $\psi$ periodic in $y_2$; the limit $d(e_1;s)$ exists and is approached
exponentially fast~\cite{MV96}.  Define the \emph{boundary density}
\begin{equation}\label{eq:rho}
  \rho(y_2;s) \;=\; \frac{1}{A_{11}}\,
  a(s,y_2)\bigl(1 - \partial_{y_1}\chi^1(s,y_2)
    + \partial_{y_1}\psi(s^-,y_2;s)\bigr),
\end{equation}
so that
\begin{equation}\label{eq:d-strip}
  d(e_1;s) = \int_0^1\chi^1(s,y_2)\,\rho(y_2;s)\dv y_2 .
\end{equation}
Separating the cell-data part from the half-space flux gives the split
\begin{equation}\label{eq:d-formula}
\begin{aligned}
  A_{11}\,d(e_1;s)
  \;=\;&
  \underbrace{\int_0^1\!\chi^1(s,y_2)\,
    a(s,y_2)\bigl(1-\partial_{y_1}\chi^1(s,y_2)\bigr)\dv y_2}_{\text{toric term}}\\
  &+\;
  \underbrace{\int_0^1\!\chi^1(s,y_2)\,
    a(s,y_2)\,\partial_{y_1}\psi\big|_{y_1=s^-}\dv y_2,}_{\text{DtN term}}
\end{aligned}
\end{equation}
established in~\cite{MV96} (Proposition~6.3) via a Green's identity argument in
the half-strip; the DtN term is the Dirichlet energy of $\psi$ and is non-negative.
The tail depends on the phase: along a sequence $\eps_k\to0$ the limit of
$\teps$ on the side is determined by $d(e_1;s_k)$ with $s_k=1/\eps_k \bmod 1$,
so it exists along subsequences with convergent phase and depends on the
subsequence~\cite{SV93,MV96}.  For instance $\eps_k=1/k$ gives $s_k=0$ for every
$k$, and the formulas above with $s=0$.

\begin{remark}[Rational vs.\ irrational normals]\label{rem:ratirr}
The phase dependence above is the general situation for a \emph{rational}
normal. The periodic data restricted to the lower dimensional boundary   $\{x/\eps : x\in\bOmega,\ n(x)=n\}$ is periodic on that boundary,
and different sequences $\eps_k\to0$ will see different parts of the original period cell.  For a smooth uniformly convex domain, or more generally a smooth domain
of finite type in the sense of~\cite{Zhuge18}, the boundary has no flat
parts and the set of boundary points with rational normal has surface
measure zero, so the normal is irrational at almost every boundary point.  For irrational
$n$ the restriction of $\{x/\eps\}$ equidistributes over the full torus $\T^d$ as
$\eps\to0$ (Weyl's theorem), the limit is \emph{unique},
and $d(n)$ is the ergodic average
\begin{equation}\label{eq:d-irrational}
  d(n) = \int_{\T^d} \chin(y)\,\rho(y,n)\dv y ,
\end{equation}
with $\rho(\cdot,n)$ determined by the half-space problem in direction
$n$~\cite{GVM12,Prange13}, which we describe in the next section.
The rational normals then do not affect the data of $\tstar$ in
$L^2(\bOmega)$.  Throughout the paper we assume that $\Omega$ is smooth and
uniformly convex or of finite type.  Then the limit $\tstar$ of $\teps$ is
unique and its boundary data is $d(n(x))\dnu$ with $d$ given
by~\eqref{eq:d-irrational}: the convergence $\teps\to\tstar$ is proved
in~\cite{GVM12} for uniformly convex domains, such as the disk used in
Section~\ref{sec:disk}, and in~\cite{Zhuge18} for domains of finite type.
\end{remark}

\subsection{The half-space problem and its lifting}
\label{sec:lifted}

Fix a unit normal $n$ and let $H_n=\{y\cdot n<0\}$ be the interior half-space
with boundary $\partial H_n=\{y\cdot n=0\}$.  Given Dirichlet
data $f$ (a $1$-periodic function on $\R^d$ restricted to $\partial H_n$), the
\emph{half-space problem} asks for the bounded $a$-harmonic function
\begin{equation}\label{eq:halfspace}
  -\nabla_y\cdot\bigl(a(y)\nabla_y\psi\bigr)=0 \;\text{ in }H_n,
  \qquad
  \psi=f \;\text{ on }\partial H_n,
  \qquad
  \psi\to\text{const} \;\text{ as }y\cdot n\to-\infty .
\end{equation}
For $n=e_1$ and $f=\chi^1$ this is the strip problem~\eqref{eq:psi} with
phase $s=0$; the tail $d(n)$ of~\eqref{eq:d-irrational} is the constant
approached at depth when $f=\chin$.  The difficulty for irrational $n$ is that
$a$ is periodic in $y$ but the boundary restriction is not periodic, so there
is no bounded cell on which to pose~\eqref{eq:halfspace}.

G\'{e}rard-Varet and Masmoudi~\cite{GVM11,GVM12} resolve this by adding a
variable.  Write each $y\in H_n$ as
\begin{equation}\label{eq:coords}
  y=\theta-tn,\qquad t:=-y\cdot n\ge0,\quad \theta:=y-(y\cdot n)\,n\in\partial H_n ,
\end{equation}
so $t$ is the distance from $y$ to the boundary hyperplane and $\theta$ is the
orthogonal projection of $y$ onto it, that is, the point of $\partial H_n$
closest to $y$.  They look for $\psi$ of the form
\begin{equation}\label{eq:gvm-def}
  \psi(y)=V(\theta,t),\qquad V \text{ $1$-periodic in }\theta\in\R^d,
\end{equation}
that is, $V$ is a function on the cylinder $\T^d\times[0,\infty)$, and $\psi$
is recovered from it by evaluating $V(\cdot,t)$ only at points $\theta$ of the
boundary hyperplane: the values of $\psi$ on the plane $\{y\cdot n=-t\}$ are
those of $V(\cdot,t)$ on $\partial H_n$, shifted by $-tn$.  Since
$\nabla_y\psi=-n\,\partial_tV+\nabla^\perp V$, where
$\nabla^\perp:=\nabla_\theta-n\,(n\cdot\nabla_\theta)$ is the gradient
in the directions tangent to the hyperplane, $\psi$ is $a$-harmonic
provided
\begin{equation}\label{eq:gvm}
  -\partial_t\bigl(a(\theta-tn)\,\partial_tV\bigr)
  -\nabla^\perp\!\cdot\bigl(a(\theta-tn)\nabla^\perp V\bigr)=0
  \quad\text{on }\T^d\times(0,\infty),
  \qquad V(\theta,0)=f(\theta),
\end{equation}
which is their equation~(2.7).  Problem~\eqref{eq:gvm} is posed on a bounded cell in $\theta$ for every $n$,
but its coefficient $a(\theta-tn)$ depends on the depth $t$.  Problem~\eqref{eq:gvm} is well posed in the energy class for every
$n$~\cite[Prop.~2]{GVM11}, and the convergence of $V(\cdot,t)$ to a constant as
$t\to\infty$ was shown for Diophantine $n$ in~\cite{GVM11}, and for every
irrational $n$, in the half-space form, by Prange~\cite{Prange13}.

We now make a further change of variables so that the coefficient no longer depends on $t$.  Set
\begin{equation}\label{eq:lift-def}
  \Psi(t,\theta):=V(\theta+tn,\,t),
  \qquad\text{equivalently}\qquad
  \psi(y)=\Psi(-y\cdot n,\,y),
\end{equation}
so that the cell variable of $\Psi$ is the point $y$ itself (mod $\mathbb Z^d$)
rather than its projection on the boundary hyperplane.  (We put the depth
variable first in $\Psi$ since $\Psi$ will be viewed as an evolution in $t$,
see Theorem~\ref{thm:riccati}.)  By the chain rule,
\begin{equation}\label{eq:lift-calc}
  \partial_tV=(\partial_t-\G)\Psi,\qquad \nabla^\perp V=\nabla^\perp\Psi,
  \qquad \G:=n\cdot\nabla_\theta ,
\end{equation}
and $a(\theta-tn)$ becomes $a(\theta)$ in the new cell variable.  Substituting
into~\eqref{eq:gvm}, $\psi$ is $a$-harmonic provided
\begin{equation}\label{eq:lifted}
  (\partial_t-\G)\,a(\theta)\,(\partial_t-\G)\,\Psi \;=\; \Lp\,\Psi
  \quad\text{on }(0,\infty)\times\T^d,
  \qquad
  \Lp:=-\nabla^\perp\!\cdot\bigl(a\nabla^\perp\,\cdot\bigr),
\end{equation}
with $\Psi(0,\cdot)=f$.  The two forms~\eqref{eq:gvm} and~\eqref{eq:lifted}
are equivalent; in~\eqref{eq:lifted} the coefficient $a(\theta)$ does not
depend on $t$, and instead the normal derivative appears in the combination
$\partial_t-\G$, with $\G$ the derivative along $n$ in the cell.  This is the
form we work with (its advantage is explained in
Remark~\ref{rem:why}).

For rational $n$ the lines $\{\theta+zn^\perp:\ z\in\R\}$ in the cell close
up after a finite length, and the lifted problem restricted to such a line is
the strip problem~\eqref{eq:psi} for one phase, so for rational $n$
problem~\eqref{eq:lifted} is the collection of the strip problems for all
phases.  For irrational $n$ each such line is dense in the cell, and there is
no strip problem to reduce to.

\subsection{The DtN operator and the boundary density}
\label{sec:dtn}

\begin{definition}
The \emph{DtN operator} $\Nn$ of~\eqref{eq:lifted} maps the boundary data
$f=\Psi(0,\cdot)$ to the conormal flux
\[
  \Nn f \;:=\; -a\,(\partial_t-\G)\Psi_f\big|_{t=0}
  \;=\; n\cdot a\nabla_y\psi_f\big|_{\partial H_n},
\]
where $\Psi_f$ is the bounded solution of~\eqref{eq:lifted} with data $f$ and
$\psi_f(y):=\Psi_f(-y\cdot n,y)$ is the corresponding $a$-harmonic function on
the half-space $H_n$; the second expression says that $\Nn f$ is the usual
conormal derivative of $\psi_f$ on the hyperplane $\partial H_n$, written in the
cell variable.
\end{definition}

Throughout, $\langle f,g\rangle=\int_{\T^d}\bar f\,g\dv\theta$ is the
$L^2(\T^d)$ inner product and $\|f\|=\langle f,f\rangle^{1/2}$; the
functions are real except in the Fourier computations of
Section~\ref{sec:computing}.  With this inner product, Green's identity on
$[0,\infty)\times\T^d$ gives
\begin{equation}\label{eq:energy}
  \langle f,\Nn g\rangle
  = \int_0^\infty\!\!\int_{\T^d}
    a\Bigl[(\partial_t-\G)\Psi_f\,(\partial_t-\G)\Psi_g
      + \nabla^\perp\Psi_f\cdot\nabla^\perp\Psi_g\Bigr]\dv\theta\dv t ,
\end{equation}
so $\Nn$ is symmetric and non-negative on $L^2(\T^d)$, with kernel the
constants when $n$ is irrational (for rational $n$ see
Remark~\ref{rem:fd}).

By the definition of $\astar$, the normal component $a(1-\partial_n\chin)$
of the cell flux has mean
\begin{equation}\label{eq:Ann}
  \int_{\T^d}a\,(1-\partial_n\chin)\dv\theta=A_{nn}=n\cdot\astar n .
\end{equation} Define the \emph{boundary density} for the normal $n$ by
\begin{equation}\label{eq:rho-cell}
  \rho(\cdot,n) \;:=\; \frac{1}{A_{nn}}\Bigl(a\,(1-\partial_n\chin)+\Nn\chin\Bigr)
  \qquad\text{on }\T^d ;
\end{equation}
it is the normal cell flux of~\eqref{eq:Ann} corrected by the conormal
flux of the boundary layer with data $\chin$.  The results of~\cite{GVM12}
and~\cite{Prange13} can now be stated in the form used below.

\begin{theorem}[G\'{e}rard-Varet--Masmoudi, Prange]\label{thm:known}
Let $n$ be irrational.  For all smooth data $f$ the lifted
problem~\eqref{eq:lifted} has a unique solution $\Psi_f$ with
$(\partial_t-\G)\Psi_f$ and $\nabla^\perp\Psi_f$ in $L^2((0,\infty)\times\T^d)$;
it is smooth, and all derivatives of $(\partial_t-\G)\Psi_f$ and
$\nabla^\perp\Psi_f$ are in $L^2((0,\infty)\times\T^d)$ as
well~\cite[Prop.~2]{GVM11}.  The solution converges to a constant at depth: for $n$ satisfying a
Diophantine condition, $\Psi_f(t,\cdot)$ converges on $\T^d$ as $t\to\infty$,
faster than any power of $t$~\cite{GVM11}; for every irrational $n$, the
half-space solution $\psi_f(y)=\Psi_f(-y\cdot n,y)$ converges as
$y\cdot n\to-\infty$, locally uniformly in the tangential variable and
without a rate~\cite[Theorem~1.2]{Prange13}.  The constant
is~\cite{Prange13}
\begin{equation}\label{eq:representer}
   \lim_{y\cdot n\to-\infty}\psi_f(y)=\int_{\T^d}\rho(\theta,n)\,f(\theta)\dv\theta .
\end{equation}
In particular $\int_{\T^d}\rho=1$ and the boundary layer tail is
$d(n)=\int_{\T^d}\rho\,\chin$, i.e.
\begin{equation}\label{eq:d-twoterm}
  A_{nn}\,d(n)=\int_{\T^d}a\,\chin(1-\partial_n\chin)\dv\theta+\langle\chin,\Nn\chin\rangle .
\end{equation}
\end{theorem}

The density~\eqref{eq:rho-cell} is the formula of~\cite[eq.~(6.4)]{Prange13}
read on the cell: Prange writes the tail as means along the boundary
hyperplane, which for irrational $n$ are integrals over $\T^d$, and the
conormal derivative of the boundary layer with data $\chin$ appearing there is
$\Nn\chin$ by definition of the DtN operator.

For rational $n$ the lifted problem contains all phases at
once, and formulas~\eqref{eq:rho-cell} and~\eqref{eq:d-twoterm} produce the average over $s$ of the strip tails
$d(n;s)$ of~\cite{MV96} (Remark~\ref{rem:rational}).  We denote this phase
average by $\bar d(n)$.  It is what the algorithm of
Section~\ref{sec:computing} returns at a rational normal, and what the tables
report there; it is not the limit of $\teps$ along any particular sequence
$\eps_k\to0$.  For irrational $n$, $\bar d(n)=d(n)$.  By a result of Feldman and Zhang~\cite[Theorem~1.2 and the discussion
following Theorem~1.1]{FZ19}, for linear divergence-form equations in any
dimension the tail of a fixed periodic datum is H\"older continuous on the
irrational normals, and its continuous extension to a rational normal is the
average over the phase of the rational-normal tails.  Their data do not
depend on $n$, but $\chin=n_j\chi^j$ and the tail is linear in the boundary data, so
the result applies to each $\chi^j$ and then to $d(n)$.  Thus $\bar d$ is
the continuous extension of $d$ to all normals.

We also state a variational description of $\Nn$, valid for every normal,
which is used in Section~\ref{sec:computing}.  For $0<T\le\infty$ and
$\Psi$ defined on $(0,T)\times\T^d$ let
\begin{equation}\label{eq:ET}
  \mathcal E_T(\Psi):=\int_0^T\!\!\int_{\T^d}
  a\Bigl[\bigl|(\partial_t-\G)\Psi\bigr|^2+\bigl|\nabla^\perp\Psi\bigr|^2\Bigr]\dv\theta\dv t ,
\end{equation}
and let $$X:=\{f\in L^2(\T^d):\ \nabla^\perp f\in L^2(\T^d)\}.$$  By the energy
identity~\eqref{eq:energy}, $\langle f,\Nn f\rangle=\mathcal E_\infty(\Psi_f)$
for the solution $\Psi_f$ of Theorem~\ref{thm:known}.

\begin{lemma}[Minimal energies]\label{lem:energy}
Let $n$ be any unit vector and $f\in X$.  For every $T\in(0,\infty]$ the
problem
\begin{equation}\label{eq:NT}
  \langle f,N_Tf\rangle:=\min\bigl\{\mathcal E_T(\Psi):\ \Psi(0,\cdot)=f\bigr\},
\end{equation}
over $\Psi$ with $(\partial_t-\G)\Psi,\nabla^\perp\Psi\in L^2((0,T)\times\T^d)$,
has a unique minimizer. 
The minimum is a quadratic form in $f$ with
\begin{equation}\label{eq:NT-bound}
  \langle f,N_Tf\rangle\le C\bigl(\|f\|^2+\|\nabla^\perp f\|^2\bigr),
\end{equation} where $C$
depends only on the bounds of $a$.  It is nondecreasing in $T$, and
$$\langle f,N_Tf\rangle\to\langle f,N_\infty f\rangle$$ as $T\to\infty$.  If $n$
is irrational and $f$ is smooth, the minimizer for $T=\infty$ is the
solution $\Psi_f$ of Theorem~\ref{thm:known}, so that $N_\infty=\Nn$ on
$C^\infty(\T^d)$.
\end{lemma}

\begin{proof}
Write $U(t,\theta):=\Psi(t,\theta-tn)$, the unknown of~\eqref{eq:gvm}.  Then
$\partial_tU$ and $\nabla^\perp U$ are $(\partial_t-\G)\Psi$ and
$\nabla^\perp\Psi$ evaluated at $(t,\theta-tn)$, and since
$\theta\mapsto\theta-tn$ preserves the measure on $\T^d$,
\[
  \mathcal E_T(\Psi)=\int_0^T\!\!\int_{\T^d}a(\theta-tn)
  \bigl[|\partial_tU|^2+|\nabla^\perp U|^2\bigr]\dv\theta\dv t .
\]
Let $a\ge\lambda>0$ and $T<\infty$.  We minimize over the Hilbert space
\[
  \mathcal H_T:=\{U\in H^1(0,T;L^2(\T^d)):\ \nabla^\perp U\in L^2((0,T)\times\T^d)\},
\]
with norm $\|U\|_{\mathcal H_T}^2:=\|U\|^2_{L^2}+\|\partial_tU\|^2_{L^2}+\|\nabla^\perp U\|^2_{L^2}$
(all on $(0,T)\times\T^d$),
restricted to the closed affine set $\{U\in\mathcal H_T:\ U(0)=f\}$, where the
trace $U(0)$ makes sense because $U\in H^1$ in $t$.  On this set
$\mathcal E_T$ is continuous, and we show it is coercive as follows: writing $U(t):=U(t,\cdot)$
as an element of $L^2(\T^d)$,
$$U(t)=f+\int_0^t\partial_sU(s)\dv s,$$ so
$$\|U\|^2_{L^2((0,T)\times\T^d)}\le2T\|f\|^2+T^2\|\partial_tU\|^2_{L^2((0,T)\times\T^d)},$$
and hence we have the coercivity inequality $$\|U\|_{\mathcal H_T}^2\le C_T\bigl(\|f\|^2+\mathcal E_T(U)/\lambda\bigr).$$
It is also strictly convex there, because $\mathcal E_T(W)=0$ with $W(0)=0$
forces $\partial_tW=0$ and hence $W=0$.  So it has a unique minimizer, and
the minimum is a quadratic form in $f$. Choosing the candidate
$U(t,\theta)=\varphi(t)f(\theta)$, with $\varphi$ smooth,
$\varphi(0)=1$ and $\varphi=0$ for $t\ge1$, gives
\[
  \langle f,N_Tf\rangle\le\mathcal E_T(U)
  \le\sup a\int_0^1\bigl(\varphi'(t)^2\|f\|^2+\varphi(t)^2\|\nabla^\perp f\|^2\bigr)\dv t ,
\]
which is the bound~\eqref{eq:NT-bound}, uniformly in $T$.
Next we let $T\to\infty$.  Write $m_\infty$ for the infimum of
$\mathcal E_\infty$ over competitors with $U(0)=f$ (finite, by the candidate
above).  For $T<S\le\infty$ the restriction to $(0,T)$ of a competitor
on $(0,S)$ is a competitor on $(0,T)$ with no larger energy, so
$\langle f,N_Tf\rangle\le\langle f,N_Sf\rangle\le m_\infty$: the minima are
nondecreasing in $T$ and bounded, hence converge to a limit
$q\le m_\infty$.  We show that $q=m_\infty$ and that the infimum is attained.
Let $U_T$ be the minimizer for depth $T$ and fix $S<\infty$.  For $T\ge S$,
the coercivity estimate on $(0,S)$ applied to the restriction of $U_T$ gives
$\|U_T\|^2_{\mathcal H_S}\le C_S\bigl(\|f\|^2+\mathcal E_S(U_T)/\lambda\bigr)
\le C_S\bigl(\|f\|^2+q/\lambda\bigr)$, since
$\mathcal E_S(U_T)\le\mathcal E_T(U_T)=\langle f,N_Tf\rangle\le q$.  So the
$U_T$ are bounded in $\mathcal H_S$ for every $S$, and along a sequence
$T_k\to\infty$ (chosen by a diagonal argument over $S=1,2,\dots$) they converge
weakly in every $\mathcal H_S$ to a function $U_\infty$ on $(0,\infty)\times\T^d$.  The
trace at $t=0$ is a bounded linear map on $\mathcal H_S$, so $U_\infty(0)=f$.
Since $\mathcal E_S$ is convex and continuous on $\mathcal H_S$, it is weakly
lower semicontinuous, and
\[
  \mathcal E_S(U_\infty)\le\liminf_k\mathcal E_S(U_{T_k})\le\liminf_k\mathcal E_{T_k}(U_{T_k})=q
  \qquad\text{for every }S .
\]
Letting $S\to\infty$ gives $\mathcal E_\infty(U_\infty)\le q\le m_\infty$,
while $\mathcal E_\infty(U_\infty)\ge m_\infty$ because $U_\infty$ is a
competitor.  So $U_\infty$ attains the infimum, $q=m_\infty=\langle f,N_\infty f\rangle$,
and the minimizer for $T=\infty$ is unique by the same strict convexity
argument as before ($\mathcal E_\infty(W)=0$ and $W(0)=0$ force $W=0$).

Finally let $n$ be irrational and $f$ smooth.  The solution $\Psi_f$ of
Theorem~\ref{thm:known} has finite energy, so it is a competitor for
$T=\infty$, and we show that it is the minimizer.  Every competitor is
$\Psi_f+W$ with $W(0)=0$ and $\mathcal E_\infty(W)<\infty$, and expanding
the square,
\[
  \mathcal E_\infty(\Psi_f+W)=\mathcal E_\infty(\Psi_f)+2B_\infty(\Psi_f,W)+\mathcal E_\infty(W),
\]
where $B_T$ denotes the bilinear form of $\mathcal E_T$.  Since
$\mathcal E_\infty(W)\ge0$, it suffices to show that $B_\infty(\Psi_f,W)=0$.
Multiply~\eqref{eq:lifted} by $W$ and integrate over $(0,T)\times\T^d$.
Integrating by parts in $t$, with $W(0)=0$, and using that $\G$ is
antisymmetric on $L^2(\T^d)$, the left side becomes
\[
  \bigl\langle W(T),\,a(\partial_t-\G)\Psi_f(T)\bigr\rangle
  -\int_0^T\!\bigl\langle a(\partial_t-\G)\Psi_f,(\partial_t-\G)W\bigr\rangle\dv t ,
\]
and the right side becomes $\int_0^T\langle a\nabla^\perp\Psi_f,\nabla^\perp W\rangle\dv t$
after integrating by parts in $\theta$.  Rearranging (first for smooth $W$,
then by density),
\begin{equation}\label{eq:BT}
  B_T(\Psi_f,W)=\bigl\langle W(T),\,a(\partial_t-\G)\Psi_f(T)\bigr\rangle .
\end{equation}
We now let $T\to\infty$ in~\eqref{eq:BT}.  The left side tends to
$B_\infty(\Psi_f,W)$, since both $\Psi_f$ and $W$ have finite energy on
$(0,\infty)$.  For the right side, set
$g(t):=\|a(\partial_t-\G)\Psi_f(t)\|$, so that by Cauchy--Schwarz
\begin{equation}\label{eq:BT-bound}
  |B_T(\Psi_f,W)|\le\|W(T)\|\,g(T).
\end{equation}
Both factors are controlled by the energy.  First, $g$ is square
integrable on $(0,\infty)$, since
$\int_0^\infty g^2\dv t\le(\sup a)\,\mathcal E_\infty(\Psi_f)<\infty$.
Second, writing $\widetilde W(t,\theta):=W(t,\theta-tn)$ as at the beginning
of the proof, $\widetilde W(0)=0$ and $\partial_t\widetilde W$ is
$(\partial_t-\G)W$ evaluated at $(t,\theta-tn)$, so
\[
\begin{aligned}
  \|W(T)\|=\|\widetilde W(T)\|
  &\le\int_0^T\|\partial_t\widetilde W(t)\|\dv t
  \le\sqrt T\,\|\partial_t\widetilde W\|_{L^2((0,T)\times\T^d)}\\
  &\le\sqrt T\,(\mathcal E_\infty(W)/\lambda)^{1/2}.
\end{aligned}
\]
Inserting both in~\eqref{eq:BT-bound},
\[
  |B_T(\Psi_f,W)|\le(\mathcal E_\infty(W)/\lambda)^{1/2}\,\sqrt T\,g(T).
\]
It remains to see that $\sqrt T\,g(T)$ is small for some arbitrarily large $T$.
If $T\,g(T)^2\ge c>0$ for all $T\ge T_0$, then $g(t)^2\ge c/t$ on
$(T_0,\infty)$, which is not integrable, a contradiction.  So there is a
sequence $T_j\to\infty$ with $\sqrt{T_j}\,g(T_j)\to0$, and along it the
right side of~\eqref{eq:BT} tends to zero.
Thus $B_\infty(\Psi_f,W)=0$, $\Psi_f$ is the minimizer for $T=\infty$, and
$\langle f,N_\infty f\rangle=\mathcal E_\infty(\Psi_f)=\langle f,\Nn f\rangle$
by~\eqref{eq:energy}.
\end{proof}

Lemma~\ref{lem:energy} defines a nonnegative DtN quadratic form $N_\infty$ on
the cell for every normal, rational or irrational, without any condition on
$n$, and at irrational normals it agrees with $\Nn$; the finite
depth problems~\eqref{eq:NT} approximate it from below.  What it does not give
is the smoothness of $\Psi_f$ and its convergence to a constant at depth, which
are used in Theorems~\ref{thm:riccati} and~\ref{thm:density}; for those we rely
on~\cite{GVM11,GVM12,Prange13}.  The proof applies verbatim to the
Fourier-truncated problem of Section~\ref{sec:marching}.

As it stands the formula~\eqref{eq:rho-cell} still requires $\Nn$, i.e.\ the
half-space problem; the next section shows that $\Nn$ satisfies an
equation on $\T^d$ alone.

\section{The cell problem for the DtN operator and the density}
\label{sec:cell}

Everything in Theorem~\ref{thm:known} is expressed through the DtN operator
$\Nn$ of the lifted problem, and computing $\Nn$ has so far meant solving the
lifted problem on the half-cylinder $(0,\infty)\times\T^d$.  This section shows
that $\Nn$ satisfies an algebraic equation on $\T^d$ alone, and that
the density has an intrinsic characterization on $\T^d$ as well.

We take the well-posedness of the lifted problem (Theorem~\ref{thm:known}) for
granted and work at the level of the resulting operators.  Throughout,
$\G$, $\Lp$ and $\Nn$ are regarded as operators acting on $C^\infty(\T^d)$: for
smooth periodic data $f$ the solution $\Psi_f$ is smooth up to $t=0$
(Theorem~\ref{thm:known}), so
$\Nn f$ is a smooth periodic function, and all
compositions and commutators below, such as $\Nn a^{-1}\Nn$ and
$\G\Nn-\Nn\G$, are well defined on $C^\infty(\T^d)$.
For irrational $n$ we do not
identify the function spaces on which $\Nn$ and the commutator extend as
bounded maps. These are nontrivial because of the small divisor phenomenon (modes $k$ with
$k\cdot n^\perp$ close to zero decay slowly into the half-space), and we do not address this question here.

\subsection{The commutator-Riccati equation}

\begin{theorem}[Cell problem for the DtN operator]\label{thm:riccati}
Let $\G=n\cdot\nabla_\theta$ be the derivative along $n$ and
$\Lp=-\nabla^\perp\!\cdot(a\nabla^\perp\,\cdot)$ the tangential part of the
operator, with $\nabla^\perp=\nabla_\theta-n\,\G$.  As operators on
$C^\infty(\T^d)$, the DtN operator $\Nn$ satisfies the algebraic equation
\begin{equation}\label{eq:CR}
  \Nn\,a^{-1}\Nn \;+\; [\,\G,\Nn\,] \;=\; \Lp,
  \qquad [\G,\Nn]:=\G\Nn-\Nn\G .
\end{equation}
Moreover the lifted solution satisfies the \emph{first-order} equation in the
depth variable
\begin{equation}\label{eq:semigroup}
  \partial_t\Psi_f=-K\,\Psi_f,\qquad K:=a^{-1}\Nn-\G ,\qquad \Psi_f(0,\cdot)=f,
\end{equation}
i.e.\ $\Psi_f(t,\cdot)=e^{-tK}f$: once $\Nn$ is known, the boundary layer is
obtained by solving a linear ODE in $t$ whose unknown is a function on the cell.
\end{theorem}

\begin{proof}
The problem below depth $s$, with data $\Psi_f(s,\cdot)$, has the same
coefficient $a(\theta)$ as the problem below depth $0$ and its solution is
$\Psi_f(s+\cdot,\cdot)$.  Hence the DtN operator at every depth is the same
operator $\Nn$, i.e.
\begin{equation}\label{eq:same-depth}
  -a\,(\partial_t-\G)\Psi_f(t,\cdot)=\Nn\Psi_f(t,\cdot)\qquad\text{for all }t\ge0,
\end{equation}
which is~\eqref{eq:semigroup} after dividing by $a$.  Substituting into~\eqref{eq:lifted} and using that
$\Nn$ does not depend on $t$,
\[
\begin{aligned}
  \Lp\Psi_f=(\partial_t-\G)\bigl(-\Nn\Psi_f\bigr)
  &= -\Nn(\partial_t-\G)\Psi_f + [\G,\Nn]\Psi_f\\
  &= \Nn a^{-1}\Nn\Psi_f+[\G,\Nn]\Psi_f .
\end{aligned}
\]
Since the data $f$ are arbitrary, \eqref{eq:CR} follows.
\end{proof}

\begin{remark}[Structure of~\eqref{eq:CR}]\label{rem:why}
The proof uses that in the form~\eqref{eq:lifted} the coefficient $a(\theta)$
does not depend on $t$, so the DtN operator is the same at every depth and the
equation is algebraic.  In the form~\eqref{eq:gvm} the coefficient moves with
$t$, and the DtN operator $M(t)$ of the problem below depth $t$ satisfies the
usual Riccati differential equation.  Since that problem is the problem below
depth $0$ with the cell shifted by $tn$, $M(t)$ is $\Nn$ conjugated by the
shift $\tau_tf(\theta):=f(\theta-tn)$, and
\[
  \dot M=M\,a(\cdot-tn)^{-1}M-L_\perp(t),\qquad
  M(t)=\tau_t\Nn\tau_t^{-1},\qquad
  \dot M=-[\G,M] .
\]
Conjugating back gives~\eqref{eq:CR}: the derivative in depth has become the
commutator with $\G$.  Without it the equation would be $\Nn a^{-1}\Nn=\Lp$,
whose solution is an operator square root (Section~\ref{sec:sqrt}).  For
coefficients periodic in depth (rational $n$), stationary Riccati equations of
this kind underlie the transparent boundary conditions of Joly, Li and
Fliss~\cite{JLF06,FJ09}; \eqref{eq:CR} extends this to quasi-periodic depth
dependence.

Multiplying~\eqref{eq:CR} by $-1$ and using $\G^{T}=-\G$ puts it in the
standard form of control theory~\cite{BIM12},
\begin{equation}\label{eq:CARE}
  \mathcal A^{T}X+X\mathcal A-X\mathcal BX+\mathcal C=0,
  \qquad X:=\Nn,\quad \mathcal A:=\G,\quad \mathcal B:=a^{-1},\quad \mathcal C:=\Lp ,
\end{equation}
with $\mathcal B$ and $\mathcal C$ symmetric and non-negative.  Of its many solutions the relevant one
is the stabilizing solution, for which $e^{t(\mathcal A-\mathcal BX)}=e^{-tK}$
does not grow: the boundary layer does not grow into the half-space (in a weak
sense, because of the constant mode and, for rational $n$, the neutral modes
of Remark~\ref{rem:fd}).  Decay in depth plays the role of decay in time, so
numerical methods for such equations apply directly to~\eqref{eq:CR}
(Section~\ref{sec:marching}); they rely on the same depth invariance.
\end{remark}

\begin{remark}[Finite dimensions]\label{rem:fd}
On a finite-dimensional space of trigonometric polynomials, the setting of
Section~\ref{sec:computing}, the truncated lifted equation is a linear system
of ordinary differential equations in $t$ with constant coefficients, and
$\Nn$ is the map from boundary values to fluxes on its bounded solutions.
For rational $n$ these include, besides the constant, the \emph{neutral}
modes $e^{2\pi ik\cdot(\theta+tn)}$ with $k\cdot n^\perp=0$: functions constant
along each line $\{\theta+zn^\perp\}$, on which the flux vanishes (the
phase-dependent tails of Remark~\ref{rem:rational}).  We do not address the
extent to which~\eqref{eq:CR} together with $\Nn\ge0$ determines $\Nn$; for
irrational $n$ in infinite dimensions the decay rates of the boundary layer
accumulate at $0$ (the small divisors).
\end{remark}

\subsection{An equation for the density}

The density~\eqref{eq:rho-cell} depends on $\chin$ and $\Nn$, both of
which live on the cell, so with Theorem~\ref{thm:riccati} it is
given by quantities on $\T^d$.  We now show that it is also characterized
on the cell without reference to $\chin$, as a conserved quantity of the 
equation~\eqref{eq:semigroup}.  Part~(ii) is the cell form of the
conservation law used in~\cite{MV96} to derive the strip
formula~\eqref{eq:d-formula}; here it follows from Theorem~\ref{thm:riccati}
and the cell equation by a direct computation, and the far-field formula~(i) follows quickly.  In~\eqref{eq:rho-cell}
the density is defined through $\chin$; parts~(ii) and~(iii) characterize it
in terms of the depth evolution alone.

\begin{theorem}[The density as a conserved quantity]\label{thm:density}
Let $\rho(\cdot,n)$ be given by \eqref{eq:rho-cell}.
\begin{enumerate}
\item[(i)] $\int_{\T^d}\rho=1$, and for all boundary data $f$ for which
  $\Psi_f(t,\cdot)$ converges in $L^2(\T^d)$ to a constant as $t\to\infty$, the limit is
  $\int_{\T^d}\rho f$, i.e.\ \eqref{eq:representer} holds.
\item[(ii)] $\rho$ satisfies the linear equation
  \begin{equation}\label{eq:kernel}
    \Nn\!\left(\frac{\rho}{a}\right)+\partial_n\rho=0 ,
  \end{equation}
  that is, $K^{T}\rho=0$ with $K$ the operator of \eqref{eq:semigroup}.
  In other words, $\langle\rho,\Psi_f(t,\cdot)\rangle$ is a conserved quantity
  of the depth equation~\eqref{eq:semigroup}: it does not depend on $t$ for any
  solution.
\item[(iii)] If $n$ is such that $\Psi_f(t,\cdot)$ converges in $L^2(\T^d)$
  to a constant as $t\to\infty$ for all data $f$ (in particular every
  Diophantine $n$, by~\cite{GVM11}), then
  $\rho\ge0$, so that $\rho$ is a probability density on the cell, and $\rho$
  is the unique solution of~\eqref{eq:kernel} with $\int_{\T^d}\rho=1$.
\end{enumerate}
\end{theorem}

\begin{proof}
(ii) Apply $\Nn a^{-1}+\G$ to $A_{nn}\rho$.  Using~\eqref{eq:CR} in the form
$\Nn a^{-1}\Nn=\Lp-[\G,\Nn]$,
\begin{align*}
 (\Nn a^{-1}+\G)\bigl[a(1-\G\chin)\bigr]&=-\Nn\G\chin+\partial_na-\partial_n(a\partial_n\chin),\\
 (\Nn a^{-1}+\G)\bigl[\Nn\chin\bigr]&=\Lp\chin-\G\Nn\chin+\Nn\G\chin+\G\Nn\chin ,
\end{align*}
and adding,
\[
  (\Nn a^{-1}+\G)(A_{nn}\rho)=\Lp\chin+\partial_na-\partial_n(a\,\partial_n\chin)
  =-\nabla\!\cdot\bigl(a(\nabla\chin-n)\bigr)=0
\]
by the cell equation~\eqref{eq:cell} (inner product with $n$).  Since $\Nn$ is
symmetric and $\G$ antisymmetric, $\Nn a^{-1}+\G=K^{T}$, and for any $\ell$
on $\T^d$,
$\frac{d}{dt}\langle\ell,\Psi_f(t,\cdot)\rangle=-\langle K^{T}\ell,\Psi_f\rangle$
by~\eqref{eq:semigroup}; so $K^{T}\ell=0$ is exactly the statement that
$\langle\ell,\Psi_f(t,\cdot)\rangle$ is conserved.

(i) By~\eqref{eq:rho-cell}, \eqref{eq:Ann}, the symmetry of $\Nn$ and
$\Nn1=0$,
\[
  A_{nn}\int_{\T^d}\rho=\int_{\T^d}a(1-\partial_n\chin)+\langle\Nn1,\chin\rangle=A_{nn},
\]
so $\int\rho=1$.  If $\Psi_f(t,\cdot)\to c$ as $t\to\infty$, the conserved
quantity of~(ii) gives
\[
  \langle\rho,f\rangle=\lim_{t\to\infty}\langle\rho,\Psi_f(t,\cdot)\rangle=c\int_{\T^d}\rho=c,
\]
which is~\eqref{eq:representer}.

(iii) By the maximum principle $c\ge0$ in~(i) whenever $f\ge0$; since now
every $f$ has a limit, $f\mapsto\int\rho f$ is a positive functional and
$\rho\ge0$.  Let $K^{T}\ell=0$, so that $\langle\ell,\Psi_f(t,\cdot)\rangle$ is
independent of $t$.  Letting $t\to\infty$ and using~(i),
\[
  \langle\ell,f\rangle=\lim_{t\to\infty}\langle\ell,\Psi_f(t,\cdot)\rangle
  =\langle\ell,1\rangle\,\langle\rho,f\rangle
  \qquad\text{for every }f,
\]
i.e.\ $\ell=\langle\ell,1\rangle\,\rho$: the conserved linear functionals
are exactly the multiples of the far-field value, and the normalization
$\int\rho=1$ selects $\rho$.
\end{proof}

Part~(i) gives a proof of the representation~\eqref{eq:representer} that is
independent of~\cite{Prange13} whenever the lifted solution converges on the
torus, in particular for every Diophantine normal~\cite{GVM11}.

\begin{remark}[Rational normals]\label{rem:rational}
The proofs of Theorems~\ref{thm:riccati} and~\ref{thm:density} use only the
structure of the lifted equation, and for rational $n$ they apply with the
strip problems of~\cite{MV96} in place of Theorem~\ref{thm:known}.  For
rational $n$, where the lifted problem is the collection of the strip problems
for all phases (Section~\ref{sec:lifted}), their content must be read with
the phase in mind.  On each line $\{\theta+zn^\perp\}$, \eqref{eq:rho-cell} is the strip
density~\eqref{eq:rho} for that phase.  In the lifted variables the solutions do not converge as $t\to\infty$: on
each line the strip solution tends to its tail $d(n;s)$, but the lines move
with depth (the cell variable of $\Psi$ is the point itself,
see~\eqref{eq:lift-def}), so $\Psi_f(t,\cdot)$ approaches a traveling
profile, constant along each line and translated by $tn$, built from the
neutral modes of Remark~\ref{rem:fd}.  The density~\eqref{eq:rho-cell} is
still conserved by Theorem~\ref{thm:density}(ii), and
$\langle\rho,\Psi_f(t,\cdot)\rangle$ is the average over the phase of the
strip tails, which for $f=\chin$ is $\bar d(n)$.  The uniqueness argument of
Theorem~\ref{thm:density}(iii) uses convergence to a constant and does not
apply; we do not address uniqueness for~\eqref{eq:kernel} at rational $n$.
The density is still non-negative, since the maximum principle argument
of~(iii) applies to each strip.  For an actual
problem at a fixed $\eps$ the phase is fixed by $\eps$, as
in~\eqref{eq:phase}, and the relevant density is the restriction
of~\eqref{eq:rho-cell} to the corresponding line.
\end{remark}

\subsection{The square root  approximation}
\label{sec:sqrt}

Dropping the commutator term in \eqref{eq:CR} gives $N a^{-1}N=\Lp$, whose
non-negative solution is the operator square root
\begin{equation}\label{eq:N0}
  N_0:=a\sqrt{a^{-1}\Lp} .
\end{equation}
In the Fourier-truncated setting of Section~\ref{sec:computing}, if
$(\lambda_k,f_k)$ are the generalized eigenpairs
\begin{equation}\label{eq:geneig}
  \Lp f_k=\lambda_k\,a\,f_k,\qquad
  \langle f_k,f_\ell\rangle_a:=\int_{\T^d}a\,f_kf_\ell\dv\theta=\delta_{k\ell},
  \qquad 0=\lambda_0\le\lambda_1\le\lambda_2\le\cdots ,
\end{equation}
then $N_0f_k=\sqrt{\lambda_k}\,af_k$, and
\begin{equation}\label{eq:d-sqrt}
  A_{nn}\,d(n)\;\approx\;
  \int_{\T^d}a\,\chin(1-\partial_n\chin)\dv\theta
  +\sum_{k}\sqrt{\lambda_k}\,\langle\chin,f_k\rangle_a^2 ,
\end{equation}
where the modes with $\lambda_k=0$ contribute nothing.  (In the continuum
the eigenvalues accumulate at $0$ for irrational $n$, and $0$ has infinite
multiplicity for rational $n$, so we state the expansion for the truncated
problem only.)
$N_0$ is the DtN operator obtained by freezing the coefficient at its boundary
values along each normal line, that is, by replacing $a(\theta-tn)$ in~\eqref{eq:gvm} by $a(\theta)$; it equals $\Nn$ exactly when $[\G,\Nn]=0$.  This is the case if $a$ is
constant along $n$, since then $a$ and $\Lp$ commute with translations
along $n$, and hence so does $\Nn$; for irrational $n$ this means that $a$ is
constant.  Note that $N_0$ does not distinguish $n$ from $-n$, and the residual of $N_0$
in~\eqref{eq:CR} is $[\G,N_0]$.   For small contrast the difference between~\eqref{eq:d-sqrt} and the exact
$d(n)$ is of third order, while $A_{nn}d(n)$ itself is of second order
(Proposition~\ref{prop:contrast}); Section~\ref{sec:dcurve} compares them
numerically.

\begin{proposition}[Small contrast]\label{prop:contrast}
Let $a=a_0+\delta b$ with $a_0>0$ constant and $b$ smooth and periodic, let
$n$ be irrational, and let $\Nn$, $N_0$ and $\chin$ be the Fourier-truncated
operators and corrector of Section~\ref{sec:marching}.  Then
$\langle\chin,\Nn\chin\rangle-\langle\chin,N_0\chin\rangle=O(\delta^3)$,
with a constant that may depend on $K$ and on $n$.
\end{proposition}

\begin{proof}
Since $\chin=0$ for constant $a$, $\chin=O(\delta)$.  For $\delta=0$ both
$\Nn$ and $N_0$ equal the DtN operator $N^{(0)}$ of the constant coefficient
$a_0$, because $\G$ and $N^{(0)}$ are both Fourier multipliers and
$[\G,N^{(0)}]=0$.  Both depend analytically on $\delta$: at $\delta=0$ the
decaying modes have decay rates $|q_k|>0$ ($q_k=2\pi k\cdot n^\perp\neq0$ for
$k\neq0$, since $n$ is irrational), so they stay separated from the imaginary
axis for small $\delta$ and $\Nn$, the flux map on these modes and the
constant, is analytic in $\delta$; and $N_0$ is analytic because the kernel of
$\Lp$ remains the constants.  Hence $\Nn-N_0=O(\delta)$, and the difference of
the two quadratic forms is $O(\delta^3)$.
\end{proof}

\noindent
There is no structural reason for the $\delta^3$ term to vanish.  To first order in $\delta$, the
corrections to $N^{(0)}$ in $\Nn$ and in $N_0$ solve the same linearized
equation, except that the commutator with $\G$ is present only for $\Nn$.  In
the Fourier basis, where $N^{(0)}$ and $\G$ are diagonal with entries
$a_0|q_k|$ and $i\omega_k$ ($q_k=2\pi k\cdot n^\perp$, $\omega_k=2\pi k\cdot
n$), the two corrections therefore differ only in the entries coupling modes
with $\omega_k\neq\omega_{k'}$, that is, through the variation of $b$ along
$n$.  The examples of Section~\ref{sec:dcurve} show the resulting difference.

\section{Computing \texorpdfstring{$d(n)$}{d(n)} from the cell problem}
\label{sec:computing}

The cell problem~\eqref{eq:CR} turns the computation of $d(n)$ into a
computation on $\T^d$:
\begin{enumerate}
\item Solve the standard cell problems~\eqref{eq:cell} to obtain $\chi^j$ and $\astar$.
\item For each boundary normal $n$, compute the toric term
  $T_1(n)=\int_{\T^d}a\chin(1-\partial_n\chin)$ from cell data, and solve the
  cell problem~\eqref{eq:CR} for $\Nn$ by \emph{Riccati doubling}
  (Section~\ref{sec:marching}): in a Fourier basis on $\T^d$ the operators
  $\G$, $\Lp$ and multiplication by $a$ are matrices, and the doubling
  iteration produces a sequence of matrices on the cell converging to the
  solution of the resulting matrix Riccati equation~\eqref{eq:CR-discrete}.  All quantities are cell-level; no
  solve on the global domain $\Omega$ and no phase averaging over rational
  approximants is needed.
\item Assemble $d(n)=(T_1(n)+\langle\chin,\Nn\chin\rangle)/A_{nn}$
  by~\eqref{eq:d-twoterm}, and if desired the density $\rho(\cdot,n)$
  by~\eqref{eq:rho-cell}.
\end{enumerate}
We first describe the discretization and the doubling algorithm
(Section~\ref{sec:marching}).  We then
introduce two test coefficients, used throughout the rest of the paper, and
validate the algorithm against direct strip solves
(Section~\ref{sec:results}); finally we examine the dependence of $d(n)$ on
the normal and measure the error of the square root  approximation
(Section~\ref{sec:dcurve}).

\subsection{Discretization and Riccati doubling}
\label{sec:marching}

We discretize the problem using the Fourier modes $e^{2\pi ik\cdot\theta}$ on $\T^2$ with
\[
  k\in\mathcal K:=\{k\in\mathbb Z^2:\ |k|_\infty\le K\},
  \qquad M:=|\mathcal K|=(2K+1)^2 ,
\]
and we identify each trigonometric polynomial $$f=\sum_{k\in\mathcal K} f_k\,e^{2\pi ik\cdot\theta}$$
with its coefficient vector $(f_k)\in\mathbb C^M$.  If $a$ is not given by a trigonometric polynomial, its Fourier coefficients $\hat a_k$  are
computed by FFT.  In this
basis the operators of the lifted problem~\eqref{eq:lifted} become the
$M\times M$ matrices
\begin{equation}\label{eq:WQA}
  \begin{aligned}
    \G&=iW, &\qquad W&:=\operatorname{diag}(2\pi k\cdot n),\\
    \nabla^\perp&=iQ_\perp, &\qquad Q_\perp&:=\operatorname{diag}(2\pi k\cdot n^\perp),\\
    && A_{kk'}&:=\hat a_{k-k'} ,
  \end{aligned}
\end{equation}
with $W,Q_\perp$ real diagonal and $A$ the Hermitian Toeplitz matrix of
multiplication by $a$, so $\Lp=Q_\perp AQ_\perp$ and the cell
problem~\eqref{eq:CR} becomes
\begin{equation}\label{eq:CR-discrete}
  NA^{-1}N+i\,(WN-NW)=Q_\perp AQ_\perp ,
\end{equation}
in which $n$ enters only through the diagonal entries of $W$ and $Q_\perp$; here $A^{-1}$ is the inverse of the matrix $A$, not the matrix of multiplication by $a^{-1}$, since the truncated flux is $A(\partial_t-iW)\Psi$.
Inner products are computed on coefficient vectors,
$\langle f,g\rangle=\sum_k\overline{f_k}g_k$.  Let $N$ be the DtN matrix of
the truncated lifted problem, that is, of~\eqref{eq:lifted} with $\Psi(t,\cdot)$
restricted to the modes $k\in\mathcal K$ and the operators replaced by the
matrices~\eqref{eq:WQA}.  By the same proof as that of Theorem~\ref{thm:riccati}, $N$ solves~\eqref{eq:CR-discrete}, and by
Lemma~\ref{lem:energy} it is the limit from below of the matrices $N_T$
of~\eqref{eq:NT}; this limit is what the doubling computes, so no uniqueness
result for~\eqref{eq:CR-discrete} is needed.  (In the language of
linear-quadratic control, $N$ is the minimal non-negative solution
of~\eqref{eq:CR-discrete}~\cite{Willems71,LancasterRodman95}.)

In the doubling algorithm, given a $t$ step size $dt$,  we compute $N_T$ for $T=dt,2dt,4dt,\dots$, using
that the truncated lifted problem on $(T,2T)$ is a copy of
the same problem on $(0,T)$, since its coefficient $a(\theta)$ does not depend
on $t$.  In the Fourier coordinates, with $\Psi(t)\in\mathbb C^M$ the
coefficient vector of $\Psi(t,\cdot)$, the energy~\eqref{eq:ET} of
Lemma~\ref{lem:energy} reads
\[
  \mathcal E_T(\Psi)=\int_0^T\Bigl[(\Psi'-iW\Psi)^{*}A\,(\Psi'-iW\Psi)
  +(Q_\perp\Psi)^{*}A\,Q_\perp\Psi\Bigr]\dv t .
\]
For $0<T<\infty$ and $u,v\in\mathbb C^M$ let
\begin{equation}\label{eq:ET-def}
  E_T(u,v):=\min\bigl\{\mathcal E_T(\Psi):\ \Psi\in H^1(0,T;\mathbb C^M),\ \Psi(0)=u,\ \Psi(T)=v\bigr\},
\end{equation}
the least energy of a slab of depth $T$ with both faces prescribed.
The minimizer depends linearly on $(u,v)$, so $E_T$ is a quadratic form in
$(u,v)$, and we denote its $M\times M$ blocks by $P_T,Q_T,R_T$:
\begin{equation}\label{eq:stack-energy}
  E_T(u,v)=u^{*}P_Tu+2\,\mathrm{Re}\,u^{*}Q_Tv+v^{*}R_Tv .
\end{equation}
These blocks are not known in closed form.  The doubling computes them: the
blocks of one slab of depth $dt$ are approximated directly
in~\eqref{eq:PQR}, and Lemma~\ref{lem:glue} below gives the blocks for depth
$2T$ from those for depth $T$.

\emph{The first slab.}  The only discretization in depth is in the first
slab.  We approximate its energy by the midpoint rule: for $u$ on the top
face and $v$ on the bottom face,
\begin{equation}\label{eq:slab}
  E_{dt}(u,v)\approx dt\Bigl[\,p^{*}A\,p+q^{*}A\,q\,\Bigr],\qquad
  p=\frac{v-u}{dt}-iW\,\frac{u+v}{2},\qquad q=iQ_\perp\,\frac{u+v}{2},
\end{equation}
that is, $\partial_t$ is replaced by a difference and $\Psi$ by the average of
the two faces.  With the diagonal matrices
\begin{equation}\label{eq:Dpm}
  D_-:=-\frac1{dt}-\frac{i}{2}W,\qquad D_+:=\frac1{dt}-\frac{i}{2}W
  \qquad(\text{so that }p=D_-u+D_+v),
\end{equation}
the blocks of~\eqref{eq:slab} are
\begin{equation}\label{eq:PQR}
\begin{aligned}
  P_{dt}&=dt\bigl[D_-^{*}AD_-+\tfrac14Q_\perp AQ_\perp\bigr],\qquad
  Q_{dt}=dt\bigl[D_-^{*}AD_++\tfrac14Q_\perp AQ_\perp\bigr],\\
  R_{dt}&=dt\bigl[D_+^{*}AD_++\tfrac14Q_\perp AQ_\perp\bigr].
\end{aligned}
\end{equation}

\emph{Gluing two slabs.}  Two slabs of depth $T$ stacked give one slab of
depth $2T$, and the DtN matrix of a slab is read off its blocks.

\begin{lemma}[Gluing]\label{lem:glue}
In the truncated problem, $N_T=P_T-Q_TR_T^{-1}Q_T^{*}$, and the blocks for
depth $2T$ are
\begin{equation}\label{eq:glue}
\begin{aligned}
  P_{2T}&=P_T-Q_T(P_T+R_T)^{-1}Q_T^{*},\qquad
  Q_{2T}=-Q_T(P_T+R_T)^{-1}Q_T,\\
  R_{2T}&=R_T-Q_T^{*}(P_T+R_T)^{-1}Q_T .
\end{aligned}
\end{equation}
\end{lemma}

\begin{proof}
$E_T(u,v)=0$ forces $\Psi'=iW\Psi$, so $\Psi(t)=e^{itW}u$ and
$v=e^{iTW}u$; hence $E_T(u,0)>0$ for $u\neq0$ and $E_T(0,v)>0$ for $v\neq0$,
and $P_T$, $R_T$ and $P_T+R_T$ are positive definite.  Since nothing is
prescribed at $t=T$ in~\eqref{eq:NT},
$u^{*}N_Tu=\min_vE_T(u,v)$, and minimizing the quadratic form over $v$ gives
the first formula.  Since $a$ does not depend on $t$, the problem on $(T,2T)$
is the problem on $(0,T)$ translated in depth, so
$E_{2T}(u,v)=\min_w\bigl[E_T(u,w)+E_T(w,v)\bigr]$, the minimum over the value
$w$ at $t=T$; it is attained at $w=-(P_T+R_T)^{-1}(Q_T^{*}u+Q_Tv)$, which
gives~\eqref{eq:glue}.
\end{proof}

Starting from the blocks~\eqref{eq:PQR} and gluing by~\eqref{eq:glue}
produces matrices $\Lambda_T=P_T-Q_TR_T^{-1}Q_T^{*}$, $T=dt,2dt,4dt,\dots$;
after the first slab every step is exact.  They differ from
$N_T$ at finite $T$, because the first slab is only approximated, but their
limit is exactly $N$, for every $dt$.  The reason is that the midpoint rule
reproduces the decaying solutions of the truncated lifted problem exactly,
with the same fluxes; only the factor by which each of them decays across one
slab is changed (Remark~\ref{rem:dt}).  The choice of $dt$ therefore affects
only the speed of convergence.  Numerically, for $dt$ between $0.05$ and $2$ the limits agree with $N$
computed from the modes of the truncated problem to $10^{-13}$ in operator
norm.  We take $dt=1$.

\begin{quote}
\textbf{Riccati doubling.}\quad
\emph{Input:} the Fourier coefficients $\hat a_k$, a unit normal $n$, the
truncation $K$, a slab thickness $dt$, a tolerance.  \emph{Output:} $d(n)$,
and if wanted the density.
\begin{enumerate}
\item[1.] Form $W$, $Q_\perp$, $A$ by~\eqref{eq:WQA}.  Solve the cell
  problems~\eqref{eq:cell} for $\chi^1,\chi^2$ as $M\times M$ linear systems,
  and set $\chin=n_1\chi^1+n_2\chi^2$, $\phi:=\hat a-iAW\chin$ (the coefficients
  of $a(1-\partial_n\chin)$, with $\hat a=(\hat a_k)$),
  $A_{nn}=\phi_0$ (its mean) and $T_1=\langle\chin,\phi\rangle$.  Form
  the slab blocks by~\eqref{eq:PQR} and set $T=dt$.
\item[2.] Repeat: compute $\Lambda_T=P_T-Q_TR_T^{-1}Q_T^{*}$ and
  $E=\langle\chin,\Lambda_T\chin\rangle$; stop if $E$ has changed by less than the tolerance (this suffices for
  $d(n)$; for the density, use the residual of~\eqref{eq:CR-discrete}); otherwise glue by~\eqref{eq:glue} and set $T\leftarrow2T$.
\item[3.] $d(n)=(T_1+E)/A_{nn}$, and
  $\rho(\cdot,n)=\bigl(\phi+\Lambda_T\chin\bigr)/A_{nn}$
  by~\eqref{eq:rho-cell}.
\end{enumerate}
\end{quote}

\begin{remark}[Why the limit does not depend on $dt$]\label{rem:dt}
As $T\to\infty$ with $dt$ fixed, the doubling converges to the DtN matrix
$\Lambda_\infty$ of an infinite stack of slabs of thickness $dt$, and the
claim is that $\Lambda_\infty=N$ exactly, for every $dt$; the step size
affects only the finite-$T$ iterates.  This is a known property of the
doubling algorithm, and we give the short direct check.  Both $\Lambda_\infty$
and $N$ are determined by their action on the decaying modes, each sending
the boundary value of a mode to its flux, so it suffices to check that the
two problems have the same decaying modes and give them the same flux.  A
solution on the infinite stack is a sequence of face values $u_0,u_1,\dots$ with finite
total energy $\sum_j\bigl(u_j^{*}Pu_j+2\,\mathrm{Re}\,u_j^{*}Qu_{j+1}+u_{j+1}^{*}Ru_{j+1}\bigr)$,
where $(P,Q,R)=(P_{dt},Q_{dt},R_{dt})$ are the blocks~\eqref{eq:PQR}.  Setting
the derivative with respect to each interior face to zero gives the
equilibrium conditions
\[
  Q^{*}u_{j-1}+(P+R)u_j+Qu_{j+1}=0,\qquad j\ge1,
\]
and the flux at the top face is $\Lambda_\infty u_0=Pu_0+Qu_1$.  Now try a
geometric sequence $u_j=z^j\varphi$ with $|z|<1$.  Substituting the
blocks~\eqref{eq:PQR}, a short computation gives
\[
  Q^{*}+z(P+R)+z^2Q=-\tfrac{dt}{4}(1+z)^2\bigl[(\mu+iW)A(\mu+iW)-Q_\perp AQ_\perp\bigr],
  \qquad z=\frac{1-\mu\,dt/2}{1+\mu\,dt/2},
\]
so $z^j\varphi$ satisfies the equilibrium conditions exactly when
$(\mu+iW)A(\mu+iW)\varphi=Q_\perp AQ_\perp\varphi$, that is, when
$e^{-\mu t}\varphi$ is a decaying mode of the truncated lifted problem
(Section~\ref{sec:dcurve}); the map $\mu\mapsto z$ sends
$\mathrm{Re}\,\mu>0$ onto $|z|<1$.  For such a mode the same substitution
gives for the flux
\[
  (P+zQ)\varphi=A(\mu+iW)\varphi=N\varphi ,
\]
the continuous flux of that mode.  So the stack and the truncated lifted
problem have the same decaying modes with the same fluxes, and
$\Lambda_\infty=N$.  (The neutral modes of Remark~\ref{rem:fd} have $\mu$
imaginary, hence $|z|=1$, and carry zero flux in both problems.)  In the
Riccati literature this is the statement that the doubling algorithm is
applied to the continuous-time equation after a Cayley transform with a free
parameter, here $2/dt$, and that the transformed equation has the same
solution for every value of the parameter, which affects only the speed of
convergence~\cite{CFL05}.  A one-sided difference in place of the midpoint
rule~\eqref{eq:slab} would not have this property, and its limit would differ
from $N$ at first order in $dt$.
\end{remark}

Iterating~\eqref{eq:glue} is the doubling
algorithm of the algebraic Riccati literature, introduced by
Anderson~\cite{Anderson78} and developed in structure-preserving form by Chu,
Fan, Lin and Wang~\cite{CFLW04}; see Bini, Iannazzo and Meini~\cite{BIM12} for
a textbook account, and Bellman and Wing~\cite{BellmanWing75} for the same
construction in the language of invariant imbedding.  On the decaying modes the coupling block $Q_T$ decays like the boundary
layer and $\|Q_{2T}\|\lesssim\|Q_T\|^2$ by~\eqref{eq:glue}, so the number of
correct digits roughly doubles at each step; on the constant mode (and the
neutral modes of Remark~\ref{rem:fd}) $Q_T$ decays only like $1/T$, but
these modes contribute nothing to $\Lambda_T$.  The change of the energy
$\langle\chin,\Lambda_T\chin\rangle$ between doublings is a practical
stopping test, not an error bound (a mode with very small transverse
frequency contributes little to the increments long before it has
converged); when the full matrix is wanted, for the density, the residual
of~\eqref{eq:CR-discrete} should be used.  In our examples the loop stops after $12$--$16$ doublings, each costing two
Hermitian solves of size $(2K+1)^2$, about a second per normal at $K=8$ on a
laptop; different normals are independent.  By contrast, reading $d(n)$ off the far field
$\Psi(T,\cdot)$ at a finite depth can converge arbitrarily slowly for
irrational $n$ violating the small divisors
condition~\cite[Theorem~1.3]{Prange13}, Newton's method
for~\eqref{eq:CR} requires at each step a Lyapunov solve that is singular at
rational normals, and the direct computation from the decaying modes of the
truncated problem (used as a check in Section~\ref{sec:dcurve}) must sort
the modes by their decay rates and invert the matrix of eigenvectors, which
is delicate when decay rates are close to zero; the doubling needs neither,
only Hermitian solves.  The doubling is implemented in the Python package
\texttt{bltail}~\cite{bltail} (function \texttt{dtn\_doubling}), which takes
the periodic coefficient $a$ and a list of normals and returns $d(n)$. The examples 
there reproduce the doubling computations in Sections~\ref{sec:results}
and~\ref{sec:dcurve}.

\subsection{Test coefficients and validation}
\label{sec:results}

We will consider two test coefficients throughout the rest of the paper.

\emph{Example 1: a smooth trigonometric coefficient.}  Our first coefficient is
\begin{equation}\label{eq:a-num}
  a(y) = 2 + 0.5\sin(2\pi y_1) + 0.7\sin(2\pi(y_1+y_2)),
\end{equation}
which varies only moderately.  The homogenized matrix is
$\astar \approx \bigl[\begin{smallmatrix}1.867&-0.066\\-0.066&1.935\end{smallmatrix}\bigr]$.
This coefficient has the symmetry
$a(-y-\tfrac12e_1)=a(y)$, from which one can check that 
$T_1(n)=\int a\chin(1-\partial_n\chin)=0$ and $d(n)=d(-n)$.
 The coefficient is a trigonometric polynomial, so its Fourier
representation is exact (the correctors and the DtN matrix are not), and the
results in what follows agree to eight digits between $K=8$ and $K=12$ and to
twelve digits between $K=12$ and $K=16$.

\emph{Example 2: a high-contrast inclusion.}  As a very different test we take
a smoothed piecewise constant coefficient with contrast $10$: two unequal
squares $[0.15,0.5]^2$ and $[0.6,0.85]\times[0.6,0.8]$ with $a=10$, in a
background $a=1$, the edges smoothed over a width $\eta=0.06$,
Figure~\ref{fig:dcurve2} left.  Explicitly,
\begin{equation}\label{eq:a-incl}
  a(y)=1+9\bigl[\omega(y_1;0.15,0.5)\,\omega(y_2;0.15,0.5)
  +\omega(y_1;0.6,0.85)\,\omega(y_2;0.6,0.8)\bigr],
\end{equation}
with the periodic window
$\omega(y;y_0,y_1)=\tfrac12\sum_{m\in\mathbb Z}\bigl[\tanh\tfrac{y+m-y_0}{\eta}-\tanh\tfrac{y+m-y_1}{\eta}\bigr]$.  This coefficient
does not have the symmetry of Example 1, and so $T_1\neq0$ and $d(n)\neq d(-n)$ in general.  Its
Fourier coefficients are computed by FFT on a $1024^2$ grid and the doubling
values converge to seven digits for $K=12$ and to four digits already at
$K=8$.  The convergence in $K$ is no slower near a rational normal of small
period: at $\theta=44.9^\circ$ and $45.5^\circ$ the values of $d(n)$ agree to
eight digits between $K=8$ and $K=20$ for~\eqref{eq:a-num}, and to six
digits between $K=8$ and $K=16$ for the inclusion.  (The truncated problem
has no small divisors; how the results converge as $K\to\infty$ for a normal
that is extremely well approximated by rationals is not addressed here.)  The homogenized matrix is
$\astar \approx \bigl[\begin{smallmatrix}1.703&0.041\\0.041&1.669\end{smallmatrix}\bigr]$.

We first validate the method at $n=e_1$, where the lifted problem is the
family of strip problems~\eqref{eq:psi} over all phases $s$ and the doubling
algorithm gives us the phase average $\bar d(e_1)$.  Table~\ref{tab:e1} compares the
DtN term $\langle\chin,\Nn\chin\rangle$ of~\eqref{eq:d-twoterm}, which
determines $\bar d(e_1)$, with the DtN term of the strip
formula~\eqref{eq:d-formula}, obtained by directly computing the strip
solutions at $16$ equally spaced phases and averaging over $s$.  We used finite
differences for Example~1 and finite volumes for Example~2.  The values agree to the accuracy of the strip
solves.

\begin{table}[htbp]
\centering
\begin{tabular}{lcc}
\toprule
 & strip solves & Riccati doubling \\
\midrule
Example 1 & $5.4233\times10^{-3}$ & $5.4234\times10^{-3}$ \\
Example 2 & $2.0448\times10^{-2}$ & $2.0447\times10^{-2}$ \\
\bottomrule
\end{tabular}
\caption{The DtN quadratic form $\langle\chin,\Nn\chin\rangle$ at $n=e_1$,
  from strip solves averaged over $16$ phases and from Riccati doubling ($K=8$
  for Example~1, $K=12$ for Example~2).
  Strip solves: finite differences (Example~1) and cell-centered finite
  volumes (Example~2), with Richardson extrapolation in the grid size.}
\label{tab:e1}
\end{table}

Figure~\ref{fig:rho} shows the density for both coefficients, at $n=e_1$ and
at the golden direction $\theta=\arctan\varphi\approx58.28^\circ$, $\varphi$
the golden ratio.  It is a smooth positive function with mean one
(Theorem~\ref{thm:density}), and it depends visibly on the normal: for the
inclusion, at $n=e_1$ the weight sits on the interiors of the inclusions,
while at the golden direction it sits on their edges and corners, where the
boundary line crosses the interface.

\begin{figure}[tbp]
\centering
\includegraphics[width=0.72\textwidth]{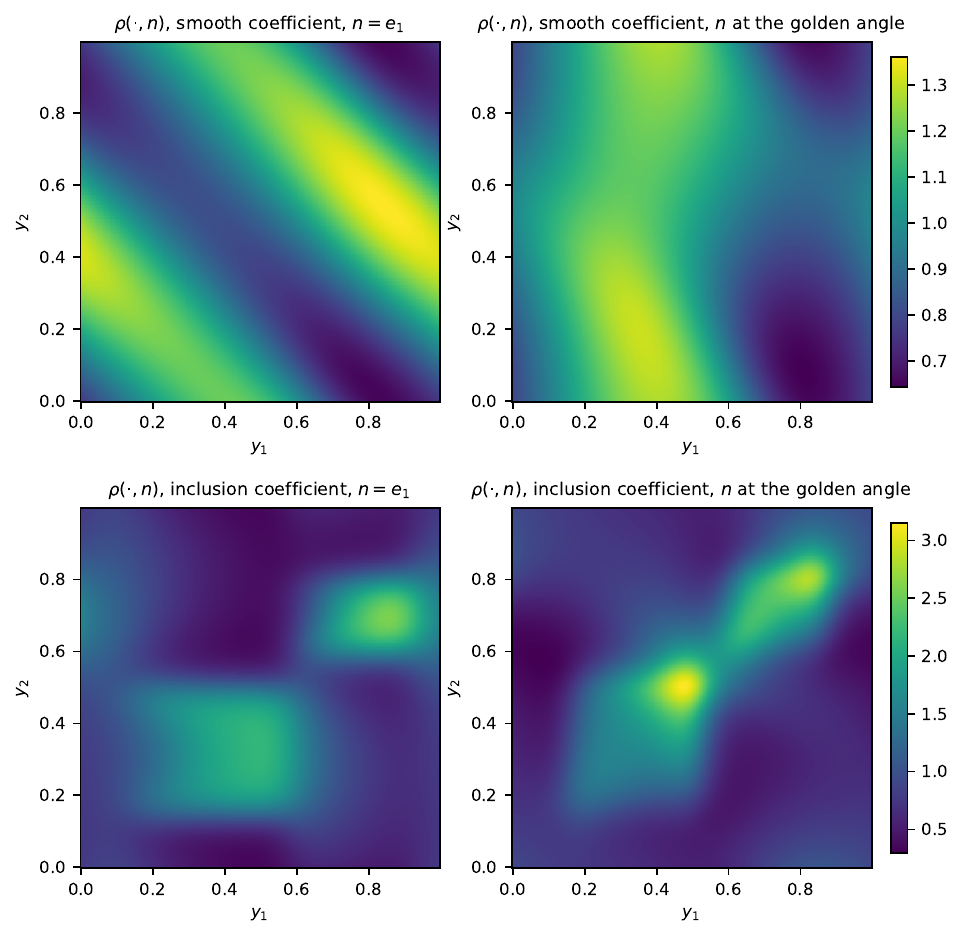}
\caption{The boundary density $\rho(\cdot,n)$ of~\eqref{eq:rho-cell} on the
  periodicity cell, for the coefficient~\eqref{eq:a-num} (top) and the
  inclusion coefficient of Figure~\ref{fig:dcurve2} (bottom), at $n=e_1$ (left)
  and at the golden direction (right); each row has its own color scale.}
\label{fig:rho}
\end{figure}

\FloatBarrier
\subsection{Dependence on the normal and the square root approximation}
\label{sec:dcurve}

Figures~\ref{fig:dcurve} and~\ref{fig:dcurve2} show $d(n)$ over the circle
of outward normals, computed by Riccati doubling, together with the
square root value~\eqref{eq:d-sqrt}: for the smooth coefficient the two
differ by a few percent at most, while for the inclusion the square root is a
poor substitute for the cell problem~\eqref{eq:CR}.  In both cases $d(n)$ is
continuous, with kinks at the rational normals of small period, consistent
with the known regularity of the
tail~\cite{GVM12,AKMP17,SZ20,FK17,FZ19}; the kinks occur at the normals
along which $a$ has a Fourier mode constant on the lines of
Remark~\ref{rem:fd}, which for~\eqref{eq:a-num} are only $0^\circ,180^\circ$
and $45^\circ,225^\circ$, the directions of its two modes $(1,0)$ and
$(1,1)$, while the inclusion has all modes.  Since $\chi_{-n}=-\chin$, the
toric term is odd, $T_1(-n)=-T_1(n)$, while the DtN term is non-negative,
so $d(n)+d(-n)\ge0$.  For~\eqref{eq:a-num}, $a(-y)=a(y+(\tfrac12,0))$, a
reflection up to a cell translation, so $d(-n)=d(n)$ (and $T_1=0$); the
inclusion has no such symmetry and its tail is clearly asymmetric.
Table~\ref{tab:golden} gives $\bar d(n)$ at the Fibonacci approximants of the
golden direction $\arctan\varphi\approx58.28^\circ$, which converge to the
golden-direction value, in agreement with the continuity of $\bar d$.

\begin{table}[htbp]
\centering
\begin{tabular}{lccc}
\toprule
$n$ & $\theta$ (deg) & $\bar d(n)$, coefficient~\eqref{eq:a-num} & $\bar d(n)$, inclusion \\
\midrule
$(1,1)$ & 45.0 & $2.351\times10^{-3}$ & $1.272\times10^{-2}$ \\
$(1,2)$ & 63.4 & $3.433\times10^{-3}$ & $1.427\times10^{-2}$ \\
$(2,3)$ & 56.3 & $3.202\times10^{-3}$ & $1.435\times10^{-2}$ \\
$(3,5)$ & 59.0 & $3.318\times10^{-3}$ & $1.443\times10^{-2}$ \\
$(5,8)$ & 58.0 & $3.278\times10^{-3}$ & $1.441\times10^{-2}$ \\
$(8,13)$ & 58.4 & $3.294\times10^{-3}$ & $1.442\times10^{-2}$ \\
$\arctan\varphi$ & 58.3 & $3.290\times10^{-3}$ & $1.442\times10^{-2}$ \\
\bottomrule
\end{tabular}
\caption{$\bar d(n)$ at the golden direction and its Fibonacci approximants
  $(p,q)\mapsto\theta=\arctan(q/p)$, for both coefficients.  For rational $n$
  the lifted problem yields the phase average $\bar d(n)$ automatically.}
\label{tab:golden}
\end{table}

\begin{figure}[htbp]
\centering
\includegraphics[width=0.9\textwidth]{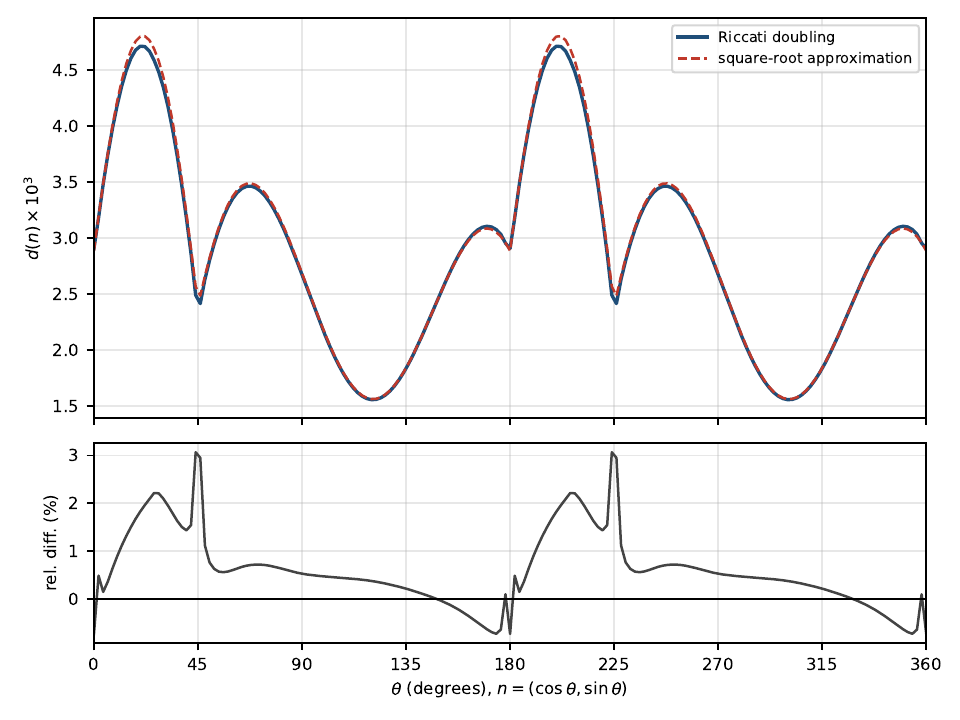}
\caption{Top: the boundary layer tail $d(n)$, $n=(\cos\theta,\sin\theta)$, for
  the coefficient~\eqref{eq:a-num}, computed by Riccati doubling (solid) and by
  the square root  value~\eqref{eq:d-sqrt} (dashed).  Bottom: relative difference.
  The kinks at $\theta=0^\circ,45^\circ,180^\circ,225^\circ$ are at the rational
  normals of smallest period; the tail is continuous there~\cite{FZ19}, and
  the kinks are what we observe numerically.  The commutator contribution is
  largest at these normals.}
\label{fig:dcurve}
\end{figure}

As an independent check of the doubling, at the golden direction we also
computed $d(n)$ a second way, without the DtN operator.  By
Theorem~\ref{thm:density}, $d(n)=\int\rho\,\chin$ is the far field, the limit
as $t\to\infty$, of the lifted solution with data $\chin$.  For the truncated
lifted problem this solution can be found exactly.  Since the coefficients do
not depend on $t$, it is a constant plus a combination of decaying
exponentials $e^{-\mu t}\varphi$, where $\mathrm{Re}\,\mu>0$ and
$(\mu+iW)A(\mu+iW)\varphi=Q_\perp AQ_\perp\varphi$.  Matching the data at
$t=0$ determines the coefficients, and the constant is the far field.  The two
values of $d(n)$ agree to a relative error of $5\cdot10^{-14}$
for~\eqref{eq:a-num} ($K=8$) and $6\cdot10^{-12}$ for the inclusion ($K=12$).

\begin{figure}[htbp]
\centering
\includegraphics[width=\textwidth]{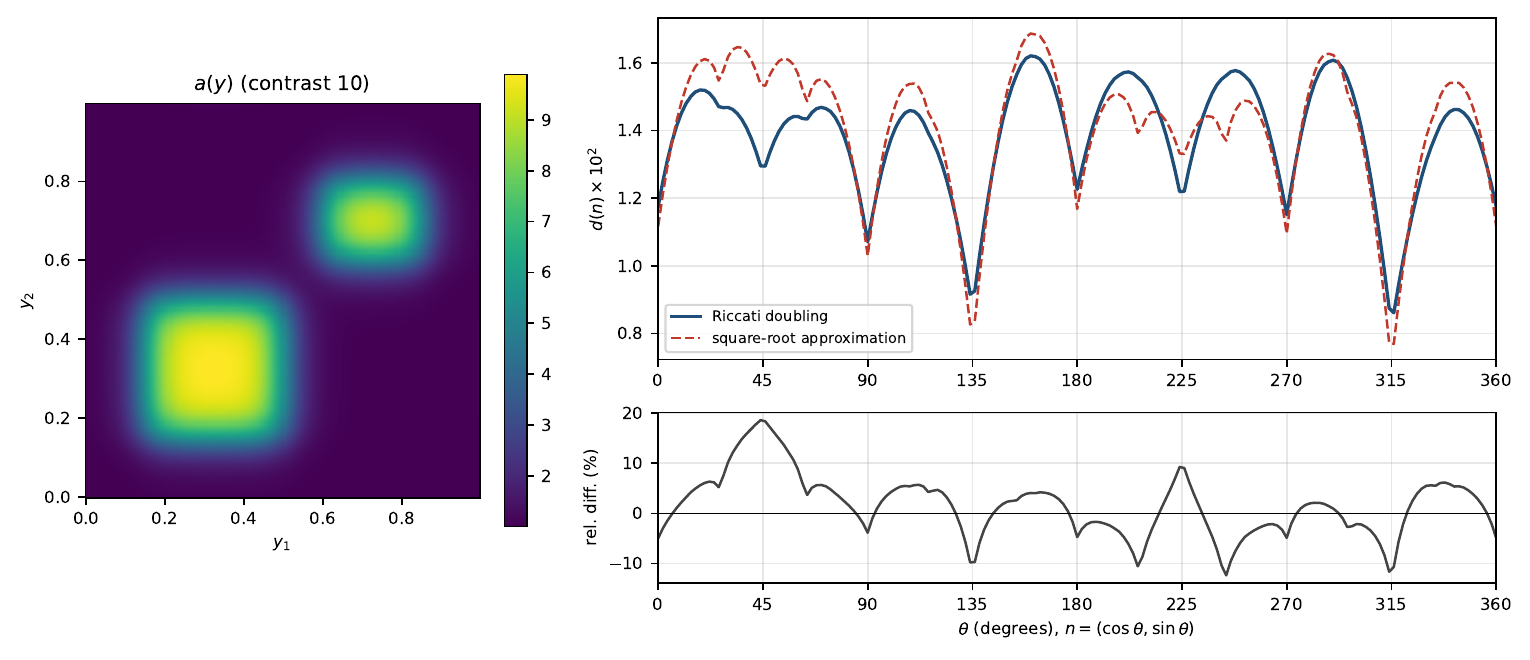}
\caption{Left: the inclusion coefficient.  Right: $d(n)$ over the circle of
  outward normals (solid) and the square root  value~\eqref{eq:d-sqrt}
  (dashed), and their relative difference.}
\label{fig:dcurve2}
\end{figure}

\begin{remark}[Piecewise constant coefficients]
Without smoothing, the Fourier--Galerkin doubling and the finite-volume strip
solves both still work, but they converge slowly: for an inclusion with
sharp edges and corners the cell corrector has limited regularity, and at the
largest resolutions we tried the two methods still differed by more than a
percent.  Efficient treatment of nonsmooth coefficients by using a discretization
of the cell that resolves the interfaces and corners requires further study.  For the questions addressed in this work, the
smoothed coefficient is adequate.
\end{remark}

\FloatBarrier
\section{The approximation \texorpdfstring{$\combo$}{tilde + star} of the boundary corrector}
\label{sec:approx}

With $d(n)$ computable for every normal, we have $\tstar$  on any smooth
domain satisfying the assumption of Remark~\ref{rem:ratirr}, but at finite $\eps$ it misses the boundary oscillation entirely.  We now
propose an approximation of $\teps$ from solves with only the
\emph{constant} operator $-\nabla\cdot(\astar\nabla\,\cdot)$.

\begin{definition}
Define the \emph{tilde corrector} $\ttilde$ as the solution of
\begin{equation}\label{eq:tilde}
  -\nabla\cdot(\astar\nabla\ttilde) = 0 \;\text{ in }\Omega,
  \qquad \ttilde = \chin(x/\eps)\,\dnu(x) \;\text{ on }\bOmega,
\end{equation}
and let $\tstar$ be the homogenized corrector~\eqref{eq:tstar}.
\end{definition}

\noindent
The proposed approximation to $\teps$ is
\begin{equation}\label{eq:combo}
  \teps \;\approx\; \combo.
\end{equation}
By linearity, $\combo$ solves
\begin{equation}
  -\nabla\cdot(\astar\nabla(\combo)) = 0 \;\text{ in }\Omega,
  \qquad
  \combo = \bigl(\chin(x/\eps) + d(n(x))\bigr)\dnu(x)
  \;\text{ on }\bOmega.
\end{equation}
No additional solve is required beyond~\eqref{eq:tilde} and~\eqref{eq:tstar}.
Both $\ttilde$ and $\tstar$ use the constant-coefficient operator $\astar$, so
neither requires a fine interior mesh to resolve oscillatory coefficients.
The only oscillations are only in the boundary data of $\ttilde$, which requires
adequate resolution along the $(d{-}1)$-dimensional boundary $\bOmega$ at scale
$\eps$, which is much cheaper than resolving $a(x/\eps)$ in the full $d$-dimensional interior.
The two terms play complementary roles as $\eps\to 0$:
\begin{itemize}
\item $\ttilde \to 0$ in $L^2(\Omega)$ as $\eps\to 0$.  Each $\chi^j$ has
  zero mean on $\T^d$, so the oscillatory data $f_\eps:=\chin(x/\eps)\dnu$
  is a sum over $k\neq0$ of terms $c_k(x)e^{2\pi ik\cdot x/\eps}$ with smooth
  amplitudes $c_k$ (the Fourier coefficients of $\chi^j$ times $n_j\dnu$),
  which decay rapidly in $k$.  On a uniformly convex boundary the phase
  $k\cdot x$ has nondegenerate stationary points on $\bOmega$ (where the
  normal is parallel to $k$), so for smooth $g$ the stationary phase bound
  gives $\int_{\bOmega}c_k\,e^{2\pi ik\cdot x/\eps}g\dv s=O((\eps/|k|)^{1/2})$,
  uniformly in the direction of $k$; summing over $k$, and using that
  $f_\eps$ is bounded in $L^2(\bOmega)$, gives $f_\eps\rightharpoonup0$
  weakly in $L^2(\bOmega)$ (on domains of finite type the stationary points
  are degenerate of finite order and the decay is slower but still holds).
  The solution operator of the constant-coefficient Dirichlet problem,
  $L^2(\bOmega)\to L^2(\Omega)$, is compact, so $\ttilde\to0$ strongly in
  $L^2(\Omega)$.
\item Under the standing assumption of Remark~\ref{rem:ratirr} (smooth,
  uniformly convex or of finite type), $\teps\to\tstar$ in $L^2(\Omega)$ as
  $\eps\to 0$~\cite{GVM12,Zhuge18}.
\end{itemize}
Therefore, under the standing assumption,
$\combo\to \tstar$ in $L^2(\Omega)$ and the approximation has the correct
$\eps\to 0$ limit (flat sides with rational normal are discussed in
Section~\ref{sec:conclusion}).  For finite $\eps$, $\ttilde$ represents the oscillatory boundary-layer
contribution while
$\tstar$ gives the bulk correction.
The boundary values of $\combo$ differ from those of $\teps$ by
$d(n)\dnu$, which is of order one, so $\|\teps-(\combo)\|_{L^2(\Omega)}$
cannot be smaller than order $\eps^{1/2}$, and the same holds for $\tstar$
alone.\footnote{Let $w=\teps-(\combo)$ and let $\Omega_\delta$ be the layer
of width $\delta$ along $\bOmega$.  The trace inequality
$\|w\|^2_{L^2(\bOmega)}\le C(\delta^{-1}\|w\|^2_{L^2(\Omega_\delta)}+\delta\|\nabla w\|^2_{L^2(\Omega_\delta)})$,
the energy bounds $\|\nabla\teps\|,\|\nabla\ttilde\|\le C\eps^{-1/2}$ (both
solve problems with the same $O(\eps^{-1/2})$ data in $H^{1/2}(\bOmega)$),
and $\|w\|_{L^2(\bOmega)}\to\|d\,\dnu\|_{L^2(\bOmega)}>0$ give, with
$\delta=\eta\eps$ and $\eta$ small, $\|w\|^2_{L^2(\Omega)}\ge c\,\eta\,\eps$
for small $\eps$.}

How fast $\ttilde\to0$ depends on the geometry of $\bOmega$, and we do not
pursue this here.  On a flat side with rational normal the situation is explicit: $\teps$
differs from its limit by an $O(1)$ oscillatory function on a strip of width
$\eps$, so $\|\teps-\tstar\|_{L^2(\Omega)}$ is of order exactly $\eps^{1/2}$;
the corresponding upper bound on a rectangle is~\cite[Lemma~3.4]{AA99}.  On
such a side $\ttilde$ does not tend to zero, since the restricted data
$\chin(x/\eps)$ has a nonzero mean along the side, and the combination must
be corrected accordingly (Section~\ref{sec:conclusion}). 

\FloatBarrier
\section{Numerical experiments on the unit disk}
\label{sec:disk}

We test the approximation on $\Omega = \{|x|<1\}\subset\R^2$ with forcing
$f = 1$, for both coefficients of Section~\ref{sec:results}.
The boundary data $d(n(x))\dnu$ of $\tstar$ is computed by Riccati doubling
(Section~\ref{sec:marching}) at $180$ equally spaced angles and interpolated
trigonometrically; $\ttilde$ and $\tstar$ use the full homogenized matrix
$\astar$, for which $u^0=(1-|x|^2)/(2\operatorname{tr}\astar)$ and
$\dnu=-1/\operatorname{tr}\astar$ on the circle.
All three correctors are computed by the finite element method (NGSolve) on the
unit disk.  The cell correctors $\chi^j$ are solved on a $64\times 64$ periodic
finite difference grid.  The oscillatory coefficient
$a(x/\eps)$ and the boundary data $\chin(x/\eps)$ are sampled on a Cartesian grid
with $96$ points per period and interpolated bilinearly; the finite element space
has order $3$ on a mesh of size $\eps/5$ for the coefficient~\eqref{eq:a-num}
(order $2$ and $\eps/10$ for the inclusion), and the reported norms are
converged with respect to both.

Figure~\ref{fig:tstar} shows the limit $\tstar$ itself for the two coefficients.
For~\eqref{eq:a-num} it is a gentle two-lobed function; for the inclusion its
boundary data $d(n(x))\dnu$ inherits the kinks of $d(n)$ at the rational normals
of small period, which appear as spots on the circle at
$0^\circ,45^\circ,90^\circ,\dots$ and decay rapidly into the interior.

\begin{figure}[htbp]
\centering
\includegraphics[width=0.95\textwidth]{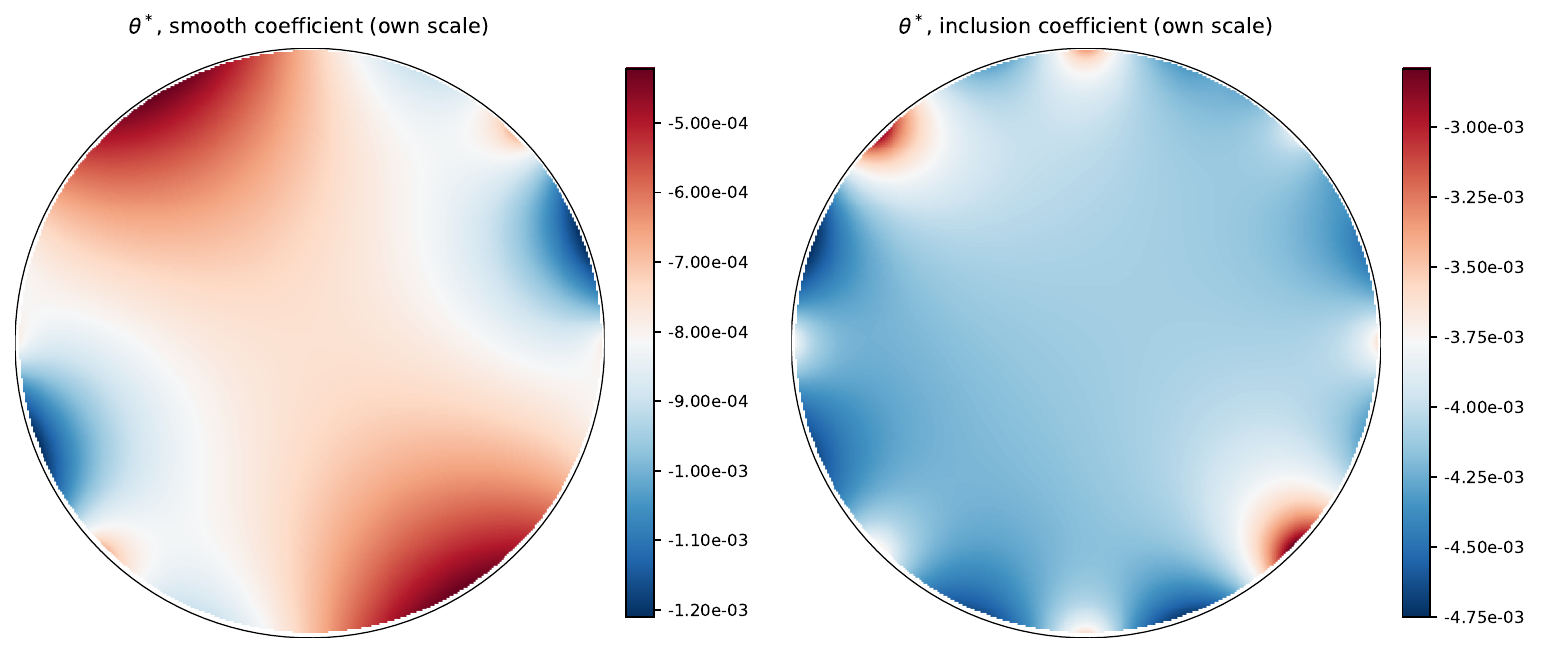}
\caption{The homogenized corrector $\tstar$ on the unit disk for the
  coefficient~\eqref{eq:a-num} (left) and for the inclusion coefficient of
  Figure~\ref{fig:dcurve2} (right), each on its own color scale.  Since
  $\dnu=-1/\operatorname{tr}\astar$, $\tstar$ is negative throughout; white
  marks the middle of the range, not zero.}
\label{fig:tstar}
\end{figure}

\emph{Example 1.}
Table~\ref{tab:norms} and Figure~\ref{fig:convergence} show that adding
$\tstar$ to $\ttilde$ reduces the error by a factor of two to four compared
with $\ttilde$ alone, and that $\|\teps-(\combo)\|$ decays at a
rate close to $\eps^{1/2}$.

\begin{table}[htbp]
\centering
\begin{tabular}{cccc}
\toprule
$\eps$ & $\|\teps\|$ & $\|\teps - \ttilde\|$ & $\|\teps - (\combo)\|$ \\
\midrule
$0.20$ & $5.79\times10^{-3}$ & $1.43\times10^{-3}$ & $5.84\times10^{-4}$ \\
$0.10$ & $4.42\times10^{-3}$ & $1.21\times10^{-3}$ & $4.45\times10^{-4}$ \\
$0.05$ & $2.91\times10^{-3}$ & $1.23\times10^{-3}$ & $3.07\times10^{-4}$ \\
\bottomrule
\end{tabular}
\caption{Coefficient~\eqref{eq:a-num}: $L^2(\Omega)$ norms for the boundary
  corrector and approximation errors on the unit disk
  ($\|\tstar\|=1.37\times10^{-3}$); the log-log slopes of
  $\|\teps-(\combo)\|$ are $0.39$ and $0.54$ over the two intervals.}
\label{tab:norms}
\end{table}

At these values of $\eps$, $\tstar$ alone gives little improvement
($\|\teps-\tstar\|\approx\|\teps\|$), and $\|\teps-\ttilde\|$ stays near
$\|\tstar\|$, since $\ttilde\to0$ while $\teps\to\tstar$.  The combination is
better than either part: $\tstar$ supplies the limit, and $\ttilde$ the
oscillation of the boundary data.

\begin{figure}[htbp]
\centering
\includegraphics[width=0.6\textwidth]{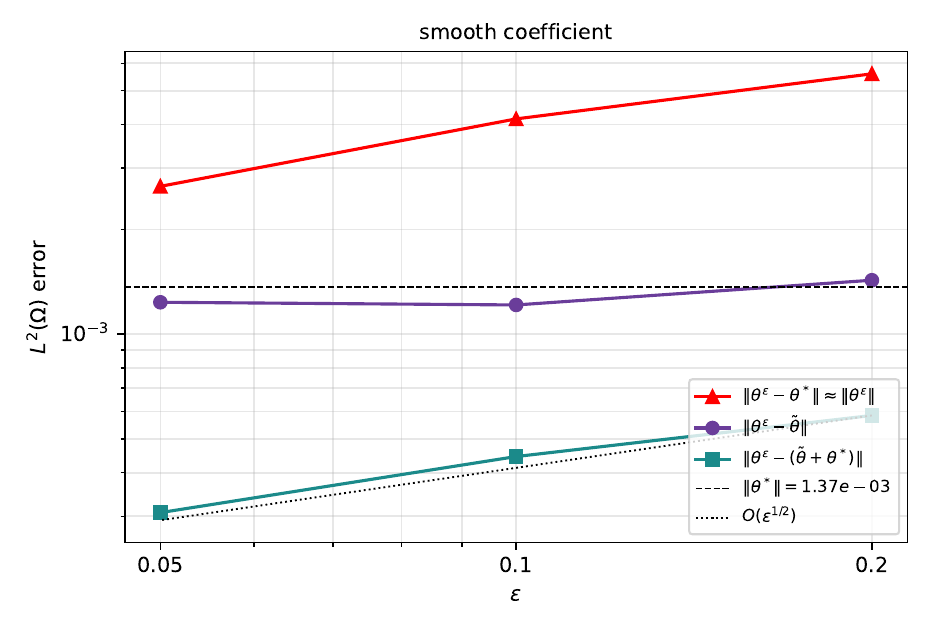}
\caption{Coefficient~\eqref{eq:a-num}: $L^2(\Omega)$ approximation errors vs.\
  $\varepsilon$ on the unit disk, for $\tilde\theta+\theta^*$ (teal squares),
  $\tilde\theta$ alone (purple circles) and $\theta^*$ alone (red triangles).
  The dashed line marks $\|\theta^*\|$.}
\label{fig:convergence}
\end{figure}

\begin{figure}[htbp]
\centering
\includegraphics[width=0.78\textwidth]{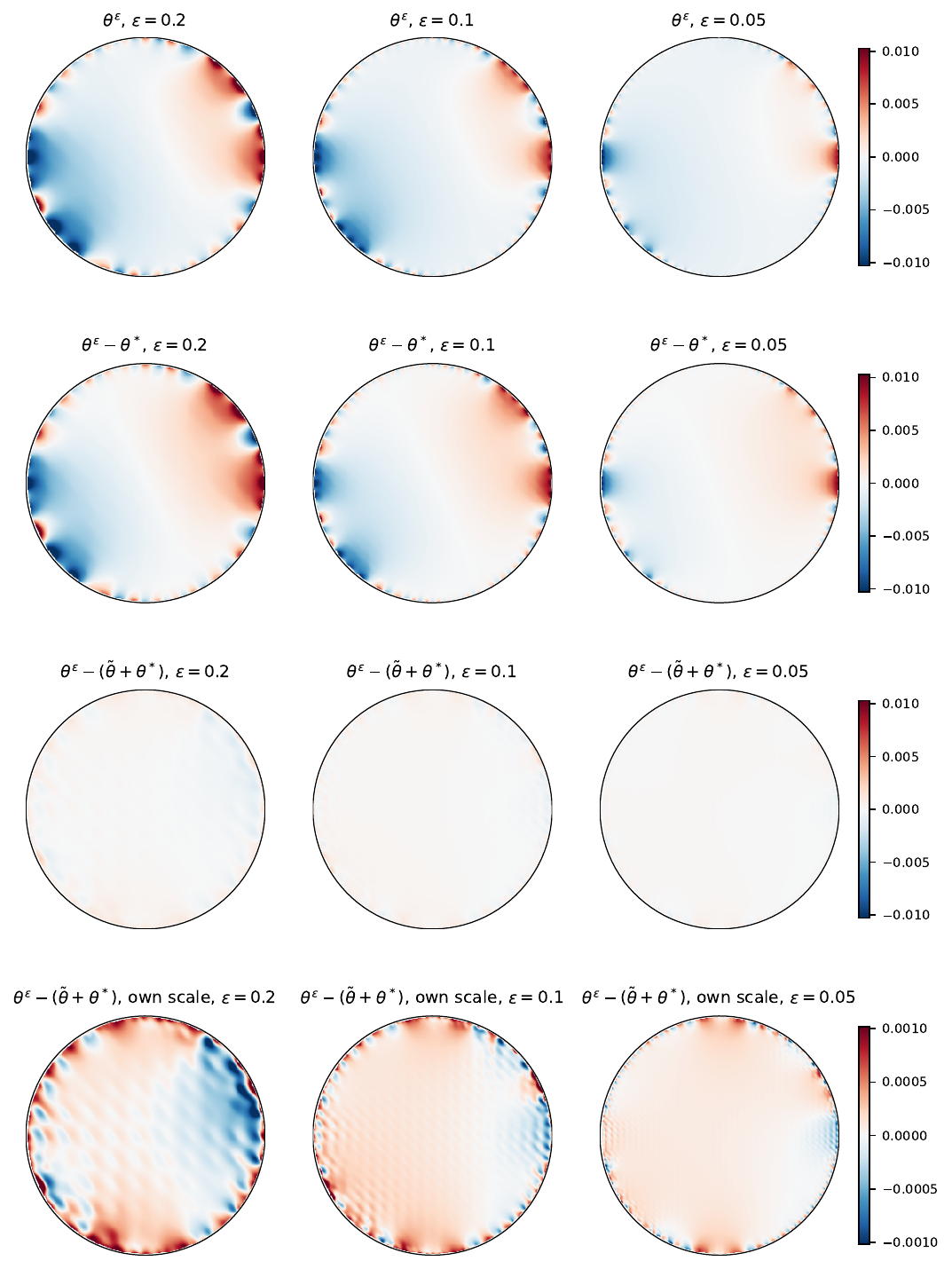}
\caption{Coefficient~\eqref{eq:a-num}: boundary corrector $\theta^\varepsilon$
  and approximation residuals on
  the unit disk for $\varepsilon = 0.20, 0.10, 0.05$ (columns left to right).
  Rows 1--3 share the same color scale ($\sim 10^{-2}$).
  \emph{Row 1:} $\theta^\varepsilon$.
  \emph{Row 2:} $\theta^\varepsilon - \theta^*$, indistinguishable from
  row 1 at this scale.
  \emph{Row 3:} $\theta^\varepsilon - (\tilde\theta+\theta^*)$ on the same
  scale.
  \emph{Row 4:} the same residual on its own scale ($\sim 10^{-3}$).}
\label{fig:panels}
\end{figure}

\emph{Example 2.}
Table~\ref{tab:norms2} and Figures~\ref{fig:convergence2}--\ref{fig:panels2}
repeat the experiment with the inclusion coefficient.  The combined
approximation again gives the smallest error, although over this range of
$\eps$ the observed decay is slower.
Restricting the error to sectors around the rational directions of small
period ($0^\circ,45^\circ,90^\circ,\dots$) suggests that the slow rate comes
mainly from there, where $d(n)$ has kinks and the boundary layer decays slowly
into the interior.

\begin{table}[htbp]
\centering
\begin{tabular}{cccc}
\toprule
$\eps$ & $\|\teps\|$ & $\|\teps - \ttilde\|$ & $\|\teps - (\combo)\|$ \\
\midrule
$0.20$ & $1.63\times10^{-2}$ & $6.75\times10^{-3}$ & $2.78\times10^{-3}$ \\
$0.10$ & $1.56\times10^{-2}$ & $7.86\times10^{-3}$ & $2.27\times10^{-3}$ \\
$0.05$ & $1.41\times10^{-2}$ & $8.32\times10^{-3}$ & $1.95\times10^{-3}$ \\
\bottomrule
\end{tabular}
\caption{Inclusion coefficient: $L^2(\Omega)$ norms and approximation errors on
  the unit disk ($\|\tstar\|=7.30\times10^{-3}$); the log-log slopes of
  $\|\teps-(\combo)\|$ are $0.29$ and $0.22$.}
\label{tab:norms2}
\end{table}

\begin{figure}[htbp]
\centering
\includegraphics[width=0.6\textwidth]{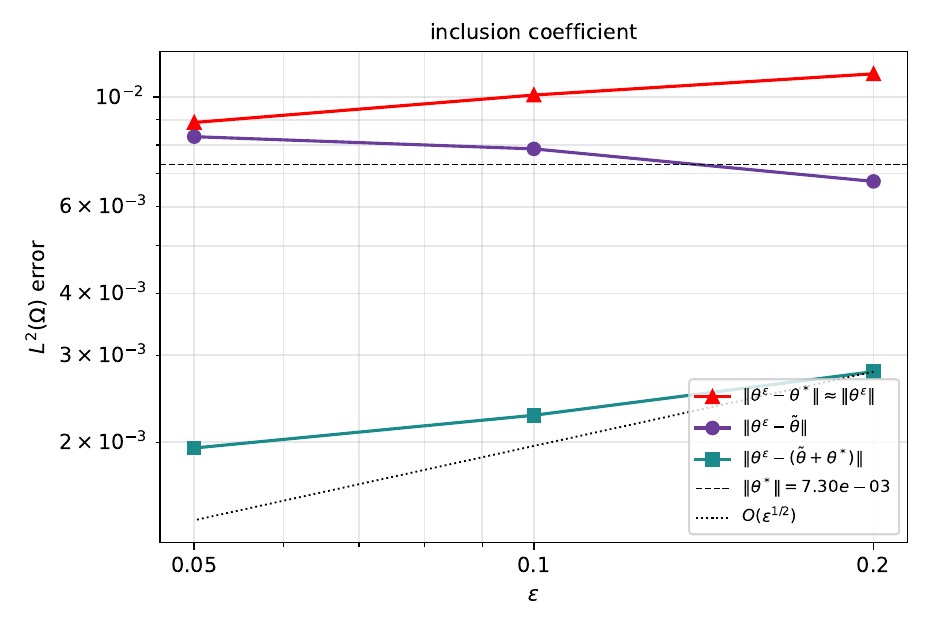}
\caption{Inclusion coefficient: $L^2(\Omega)$ approximation errors vs.\ $\eps$
  on the unit disk, as in Figure~\ref{fig:convergence}.}
\label{fig:convergence2}
\end{figure}

\begin{figure}[htbp]
\centering
\includegraphics[width=0.78\textwidth]{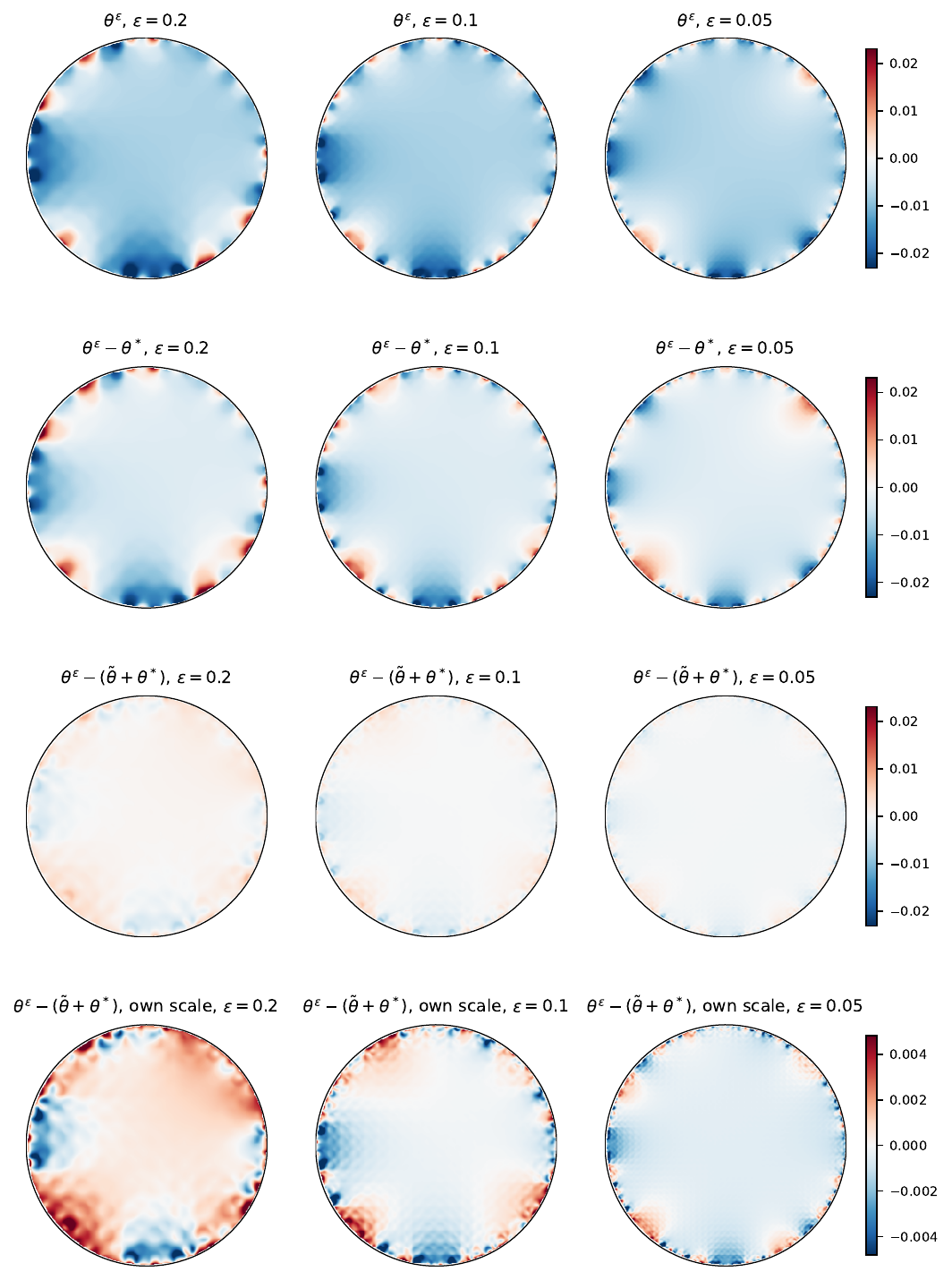}
\caption{Inclusion coefficient: $\teps$ and approximation residuals on the unit
  disk for $\eps=0.20,0.10,0.05$, with the same layout as Figure~\ref{fig:panels}
  (rows 1--3 on a common scale, row 4 on its own scale).}
\label{fig:panels2}
\end{figure}

\section{Discussion}
\label{sec:conclusion}

We have shown (Theorem~\ref{thm:riccati}) that the DtN operator of the lifted
problem satisfies the Riccati equation~\eqref{eq:CR} with a commutator on a
period cell. Therefore, for smooth uniformly convex domains, and more generally domains of finite type, the limiting boundary data of the
boundary corrector can be computed from cell problems only. 
We have also shown that the corresponding density satisfies the kernel
equation~\eqref{eq:kernel}. 
The solution to~\eqref{eq:CR} can be computed by a Riccati doubling algorithm, well known in the control theory community, 
and therefore $\tstar$ is available by solving
one cell problem per normal.  These results are stated for operators acting on
smooth periodic functions, taking the well-posedness of the lifted problem
from~\cite{GVM12}; the function spaces on which $\Nn$ and the commutator extend
for irrational normals are left open.

Since the density represents the far field for all boundary data, the same
cell problem gives the boundary layer tails in two other cases.  First, for
nonhomogeneous Dirichlet data $u^0=g$ on $\bOmega$, the cell corrector term has a tangential
part, and the boundary data of $\tstar$ becomes
$d(n)\,\dnu+d_\tau(n)\,\partial_\tau u^0$ with
$d_\tau(n)=\int_{\T^d}\rho(\cdot,n)\,\chi_\tau$ and $\chi_\tau=\tau_j\chi^j$ the
cell corrector in the tangential direction $\tau$.  The same cell problem
provides $d_\tau$, and \texttt{bltail} computes it.  Second, Neumann and
transmission conditions lead to half-space problems with a flux condition.
In two dimensions the Neumann problem reduces by duality to the Dirichlet
problem for the rotated coefficient $\tilde a=Q^{T}a^{-1}Q$: the stream
function of the Neumann corrector is a Dirichlet corrector for $\tilde a$
whose data is the normal cell corrector of $\tilde a$ times the tangential
component of the homogenized flux~\cite[eq.~(47)]{MVneumann}, so the Neumann
tails are the Dirichlet tails $d$ of $\tilde a$, which for scalar $a$ is
$1/a$ (for inhomogeneous flux data the tangential tails $d_\tau$ of
$\tilde a$ enter as well).  This was carried out
for polygons with rational sides in~\cite{MVneumann},
where the extension to irrational sides and smooth domains was left open.
For smooth strictly convex domains, homogenization of Neumann problems with
oscillating first-order boundary data was later established, with sharp
rates, by Shen and Zhuge~\cite{SZ18}.
Through this reduction the results for the
Dirichlet problem carry over to the
Neumann problem in two dimensions, and the cell problem makes its tails
computable.  In three
dimensions the flux problems require the far field of flux data, which is
governed by the inverse of $\Nn$ on mean zero functions and is again
given by the cell problem, but with a different tail formula.  This is the subject of future work. 

We also presented an approximation to the boundary corrector, $\combo$, which uses only constant-coefficient solves.
We excluded flat sides with rational normal, since there the limit depends on the
phase of the side relative to the cell and requires a strip solve at that
phase, as in Section~\ref{sec:strip}. In this case, given the phase, $\ttilde$ could be taken the same as for smooth domains, and $\tstar$ modified on the flat side to have the correct limit $d(n;s)$ minus the average of the restricted data (which is not zero in this case). 
 A flat side with irrational normal has the unique limit $d(n)$~\cite{GVM12}, but at finite $\eps$ the boundary data
must have room to equidistribute along the side: the mode $k$ of $\chin$
oscillates along it with wavelength $\eps/|k\cdot n^\perp|$, so the side must
be long enough to capture this. Near a rational
direction of small period the side may behave like a rational one
with a phase.  We have not studied this regime, nor corners.  Extending the doubling to three dimensions
(where the cell problem is on $\T^3$ and the matrices are larger but the
algorithm is unchanged), to matrix coefficients (where the lifted equation
acquires cross terms between $\partial_t-\G$ and $\nabla^\perp$, and
\eqref{eq:CR} the corresponding terms) and to systems, and using the density $\rho(\cdot,n)$
itself for instance to study the regularity of $d(n)$ in $n$ at rational
directions (see~\cite{FK17,FZ19,AKMP17,SZ20}), are also future work.

\section*{Declarations}

\paragraph{Funding.}
The author was partially supported by NSF Grant DMS-2308200.

\paragraph{Use of AI.}
This work was carried out with substantial assistance from AI systems:
ChatGPT (OpenAI) in the early stages, and Claude (Anthropic) in the later
stages.
The author posed the problem of expressing the boundary layer tail $d(n)$
through quantities on the period cell alone, by looking for the
Dirichlet-to-Neumann operator as an operator square root.  The AI
assistants proposed a Riccati equation for this operator; the first such
equation was incorrect, and the correct cell equation emerged after several
iterations, through the change of variables~\eqref{eq:lift-def}, which came
from Claude, as did the use of the doubling algorithm to solve it.  Claude
also contributed to the derivations and proofs, wrote the code, ran the
numerical experiments, assisted with the literature search and the
bibliography, and drafted parts of the text.  All figures were produced by
code, not by generative AI; the code is publicly available (see below) and
was validated against independent computations (Section~\ref{sec:results}).
The author directed the work, checked all derivations and computations,
revised the text, and assumes responsibility for all content.

\paragraph{Code availability.}
The Python package \texttt{bltail} (MIT license), which implements the
Riccati doubling of Section~\ref{sec:marching} and the square root 
approximation and reproduces the doubling values and figures of
Sections~\ref{sec:results} and~\ref{sec:dcurve}, is available at \url{\codeurl}.

\FloatBarrier

\end{document}